\documentclass[11pt]{article}
\usepackage{xcolor}
\usepackage{authblk}
\usepackage[title]{appendix}

\usepackage{pst-all}
\usepackage{pgfplots}
\usepackage{pst-plot}
\usetikzlibrary{calc}

\usepackage{xcolor}
\usepackage{tikz}
\usepackage{tkz-euclide}

\usepackage{bbm}
\usepackage{hyperref}
\hypersetup{
    colorlinks=true,              
    breaklinks=true,              
    urlcolor= black,              
    linkcolor= blue,              
    bookmarksopen=false,
    filecolor=black,
    citecolor=blue,
    linkbordercolor=blue}

\usepackage{chngcntr}

\usepackage{mathrsfs}
\usepackage{nicefrac}

\usepackage[margin=1in]{geometry}

\usepackage{amsmath,amsthm,amssymb}

\usepackage{bbm}

\usepackage{graphicx}
\usepackage{faktor}

\newcommand{\R}{\mathbb{R}}

\newcommand{\N}{\mathbb{N}}

\newcommand{\Z}{\mathbb{Z}}

\newcommand{\cov}{\normalfont \text{Cov}}
\newcommand{\vol}{\normalfont \text{Vol}}

\newcommand{\pr}{\normalfont \text{pr}}

\newcommand{\GL}{\normalfont \text{GL}}

\newcommand{\ka}
{\kappa}

\newcommand{\al}
{\alpha}

\newcommand{\be}
{\beta}

\newcommand{\SL}{\normalfont \text{SL}}

\newcommand{\SO}{\normalfont \text{SO}}

\newcommand{\Ad}{\normalfont \text{Ad}}

\newcommand{\s}{\normalfont \textbf{s}}

\newcommand{\st}{\sqrt{t}}

\newcommand{\e}{\epsilon}

\newcommand{\de}{\delta}
\newcommand{\De}{\Delta}
\newcommand{\la}{\lambda}

\newcommand{\La}{\Lambda}

\newcommand{\Ga}{\Gamma}

\newcommand{\Sph}{\mathbb{S}}

\newcommand{\V}{\mathcal{V}}

\newcommand{\Span}{\normalfont \text{Span}}

\newtheorem{theorem}{Theorem}[section]

\newtheorem{lemma}[theorem]{Lemma}

\newtheorem{proposition}[theorem]{Proposition}

\newtheorem{corollary}[theorem]{Corollary}

\theoremstyle{definition}

\theoremstyle{remark}

\newtheorem{remark}[theorem]{Remark}

\numberwithin{equation}{section}

\usepackage{setspace}

\usepackage[
style=alphabetic,
sorting=nty,
maxnames=99,
maxalphanames=99,
firstinits=true
]{biblatex}
\renewbibmacro{in:}{}
\DeclareFieldFormat[article]{citetitle}{\mkbibemph{#1}}
\DeclareFieldFormat[article]{title}{\mkbibemph{#1}}
\DeclareFieldFormat[article]{journaltitle}{#1}

\usepackage{enumerate}

\title{Limiting distribution of dense orbits in a moduli space of rank $m$ discrete subgroups in $(m+1)$-space}
\author[1]{Michael Bersudsky\thanks{\href{mailto:bersudsky.1@osu.edu}{bersudsky.1@osu.edu}}}
\author[1]{Hao Xing\thanks{\href{mailto:xing.211@osu.edu}{xing.211@osu.edu}}}
\affil[1]{Department of mathematics, the Ohio State University, 100 Math Tower, 231 W 18th Ave, Columbus, OH 43210, USA}
\begin{document}
\date{}
\maketitle
\begin{abstract}
We study the limiting distribution of dense orbits of a lattice subgroup $\Gamma\leq\SL(m+1,\mathbb{R})$ acting on $H\backslash\SL(m+1,\mathbb{R})$, with respect to a filtration of growing norm balls. The novelty of our work is that the groups $H$ we consider have infinitely many non-trivial connected components. For a specific such $H$, the homogeneous space $H\backslash G$ identifies with $X_{m,m+1}$, a moduli space of rank $m$-discrete subgroups in $\mathbb{R}^{m+1}$. This study is motivated by the work of Shapira-Sargent who studied random walks on $X_{2,3}$. 
\end{abstract}

\section{Introduction}

%The interest in the shape of rank $k$ sublattices in $\R^d$ dates back to the work of Schmidt \cite{Schmidt1998TheDO}. A refining problem is studied by \cite{emss, AES_lat}. As we described below, the space $X_{m,m+1}$ is a homogeneous space of $\SL(m+1,\R)$ and it is natural to study the dynamics of the subgroups of $\SL(m+1,\R)$. 

%\cite{Sargent2017DynamicsOT},\cite{gorodnik2022stationary}, \cite{Oh05}.

%The case when $H$ is discrete done by \cite{Oh05} and generalized by \cite{Gorodnik2004DistributionOL}.

%The case when $H$ is algebraic were studied in generality by \cite{GorodnikNevo2012} \cite{Gorodnik2003LatticeAO} 

%Recall that the space of unimodular lattices in $\R^{m}$ can be identified with $G/\Ga:=\SL(m+1,\R)/\SL(3,\Z)$, for the action of matrix on column vectors on the left (or $\Ga\backslash G$ for the action of matrix on row vectors on the right).
A well studied topic in homogeneous dynamics concerns the limiting distributions of \emph{dense} orbits of a lattice $\Ga\leq G$ acting on a homogeneous space $H\backslash G$. 
An notable basic result in this line of research is the work of Ledrappier \cite{Led99} which gives the limiting distribution of the $\Ga\leq\SL(2,\R)$-dense orbits in $ \begin{bmatrix}
    1&0\\\R&1
\end{bmatrix}\backslash \SL(2,\R)\cong\R^2\smallsetminus \textbf{0}$. An important aspect of \cite{Led99} is the method of reducing the question on the limiting distribution of $\Ga$ orbits in $H\backslash G$ to a question on $H$ orbits in $G/\Ga$. This kind of idea, which we refer to as the duality principle, has appeared extensively in the literature, and we refer the reader to \cite{GN14duality} for a detailed exposition. In this paper, we continue the above line of research, and the main novelty of our work is that the subgroup $H\leq \SL(m+1,\R)$ we consider has \emph{infinitely} many non-trivial connected components. In this regard, we note that general results for a connected $H$ were obtained in \cite{Gorodnik2004DistributionOL,GN14duality}, and results for $H$ being a lattice were given in \cite{Oh05,Gorodnik2004DistributionOL}.  Our method will also exploit the duality principle through a general theorem of \cite{Gorodnik2004DistributionOL}. 

We now proceed to describe our results, and then we will explain our motivation to consider this particular setting. In what follows, $G:=\SL(m+1,\R)$, $\Ga\leq G$ is a lattice, and  \begin{equation}\label{eq: def of general H}
H:=\left \{\begin{bmatrix}
t^{-\frac{1}{m}}q & 0 \\
v &  t 
\end{bmatrix}:t> 0, q \in \De, v \in \R^m \right\},
\end{equation}where $\De\leq\SL(m,\R)$ is a lattice. In this setting, we have the following observation:
\begin{proposition}
 For each $x_0\in H\backslash G$, the orbit $x_0\Ga$  is dense.   
\end{proposition}
This statement follows from our Theorem \ref{equidistribution result on  G mod H}, but we briefly sketch a simple proof  --  A $\Ga$-orbit passing through $Hg$ is dense in $H\backslash G$ if and only if the ``dual" $H$-orbit passing through $g\Ga$ is dense in $G/\Ga$. It's well known that each orbit of the connected component of $H$ is dense, see e.g. Proposition 1.5 of \cite{Dani1980OrbitsOE}.  

We will be considering the filtration of the dense orbits given by the Hilbert-Schmidt norm: $$\|g\|=\sqrt{\text{Trace}(^tgg)}=\sqrt{\sum_{ij}g_{ij}^2},~g\in G.$$Namely, we denote \begin{equation}\label{eq:def of Ga T}
 \Ga_T:=\{\gamma \in \Ga:\|\gamma\|\leq T\},   
\end{equation}
and we consider the following probability measures on $H\backslash G$:
$$\mu_{T,x_0}:=\frac{1}{\#\Ga_T}\sum_{\gamma\in\Ga_T}\de_{x_0.\gamma},~T>0,x_0\in H\backslash G.$$

We will prove that the above measures converge as $T\to\infty$ to the explicit measure   $\tilde{\nu}_{x_0}$ which we describe now.
Consider the  group $P$ containing $H$, which is given by\begin{equation*}
    P:=\left \{\begin{bmatrix}
t^{-\frac{1}{m}}\eta & 0 \\
v &  t 
\end{bmatrix}:t> 0, \eta \in \SL(m,\R), v \in \R^m \right\}.
\end{equation*} The space $H\backslash G$ is naturally a fiber-bundle over $P\backslash G$ with respect to the natural map sending $$\pi(Hg):=Pg,$$
and the fibers are isomorphic to $\Delta\backslash\SL(m,\R)$. We note that the base space $P\backslash G$ is identified with $\Sph^{m}$ the unit sphere in $\R^{m+1}$ via the orbit map of the right action $w.g:=\frac{w (^tg^{-1})}{\|w (^tg^{-1})\|},~g\in G,~w\in\Sph^{m}$. Thus we may view $H\backslash G$ as a fiber bundle over the sphere. Notice that $\SO(m+1,\R)$ acts transitively on $P\backslash G$ (with respect to the natural action), and we denote by $\mu_{P.\SO(m+1,\R)}$ the right $\SO(m+1,\R)$-invariant probability on $P\backslash G=P.\SO(m+1,\R)$. Now we define measures on each fiber $$\pi^{-1}(P\rho)=HP\rho,~\rho\in\SO(m+1,\R).$$ Using the Iwasawa decomposition, we consider $$x_0=:H \begin{bmatrix} \text{g}_0  & 0 \\ \textbf{v}_0 & \nicefrac{1}{\det(\text{g}_0)} \end{bmatrix}\rho_0,$$where $|\det(\text{g}_0)|=1$ and $\rho_0\in\SO(m+1,\R)$, and we define the function \begin{equation}\label{ef:def of Phi_x_0}
    \Phi_{x_0}\left(H\begin{bmatrix} \eta  & 0 \\ v & 1 \end{bmatrix}\rho\right)=\sum_{q\in \Delta}\frac{1}{\|\text{g}_0^{-1}q \eta\|^{m^2}}.
\end{equation}
Here $\rho\in\SO(m+1,\R)$ and $\eta\in \SL(m,\R)$. The expression on the right is well-defined as the Hilbert-Schmidt norm is bi-$\SO(m+1,\R)$ invariant. The infinite sum is convergent due to Lemma \ref{lemma on the integral and summation in a ball} below. We naturally identify the ``standard" fiber $\pi^{-1}(P)=H\backslash P$ with $\Delta\backslash \SL(m,\R)$ by mapping $\De\eta\mapsto H\begin{bmatrix} \eta  & 0 \\ 0 & 1 \end{bmatrix}$, which allows to define the $\SL(m,\R)$-invariant measure $\mu_{H.\SL(m,\R)}$  on $H\backslash P$ with the choice of normalization such that the measure \begin{equation}\label{eq:def of measures on fibers}
    \nu_{x_0,\pi^{-1}(P)}(f):=\int_{\pi^{-1}(P)}f(y)\Phi_{x_0}(y)d\mu_{H.\SL(m,\R)}(y),~f\in C_c(H\backslash G)
\end{equation} is a probability measure. On any other fiber $\pi^{-1}(P\rho),~\rho\in\SO(m+1,R)$, we define the pushed measure $$\nu_{x_0,\pi^{-1}(P\rho)}:=\rho_*\nu_{x_0,\pi^{-1}(P)},$$where the pushforward is via the right translation by $\rho.$ Having defined the above measures on the base space and the fibers, we can now define our probability measure on $H\backslash G$ to be \begin{equation}\label{eq:def of nu tilde x0}
    \tilde\nu_{x_0}(f)=\int_{P.\SO(m+1,\R)}\nu_{x_0,\pi^{-1}(b)}(f)d\mu_{P.\SO(m+1,\R)}(b),~f\in C_c(H\backslash G).
\end{equation}
\begin{theorem}\label{equidistribution result on  G mod H}
 Let $\Ga\leq G$ be a lattice and fix ${x_0}\in H\backslash G$. Then, $\mu_{T,{x_0}}$ converges in the weak-* topology to $\tilde\nu_{x_0}$ as $T\to\infty$. 
\end{theorem}

\subsection{Main motivating example - unimodular $m$-lattices in $(m+1)$-space}
The main motivation behind this project was to study the limiting distribution of orbits of a lattice $\Ga\leq G=\SL(m+1,\R)$ acting on a moduli space of discrete groups. Such a research direction was originally suggested by U. Shapira as natural continuation of the work \cite{Sargent2017DynamicsOT}  which considers random walks on the space of rank 2 discrete subgroups in $\R^3$ (see also the more recent work in this direction \cite{gorodnik2022stationary} which generalizes \cite{Sargent2017DynamicsOT}). We now describe our result in this setting, which is a particular case of Theorem \ref{equidistribution result on  G mod H} with $H$ as in \eqref{eq: def of general H} such that $\De=\SL(m,\Z)$.
% We now describe in geometric terms a homogeneous space $H\backslash \SL(m+1,\R)$ where $H$ is as in \eqref{eq: def of general H} with $\De=\SL(m,\Z)$. belonging  with In this paper we study the limiting distribution of dense orbits of a lattice $\Ga\leq\SL(m+1,\R)$ in the space $X_{m,m+1}$ of normalized $m$-dimensional discrete subgroups of $\R^{m+1}$ with respect to a filtration given by growing norm balls (see precise definitions below).

In what follows, $m$ is a natural number strictly larger than 1.
%We note that the space $X_{m,m+1}$ is a "hybrid" of the more familiar space $X_m:=\SL(m,\R)/\Delta$ of unimodular lattices in $\R^m$  with the grassmanian of hyperplanes in $\R^{m+1}$. The space of lattices $X_m$ plays a fundamental role in many of the applications of homogeneous dynamics, and therefore we find it interesting to consider the moduli space $X_{m,m+1}$. 
We say that $\La\subset\R^{m+1}$ is an $m$-lattice if $\La$ is the $\Z$-Span of a tuple of linearly independent vectors  ${v_1,v_2,...,v_m}\in\R^{m+1}$, that is, $$\Lambda:=\text{Span}_{\Z}\{v_1,v_2,...,v_m\}.$$ We define the covolume function:$$\cov(\Lambda):=\sqrt{\det(\langle v_i,v_j\rangle)},$$
which is the area of a fundamental parallelogram of $\Lambda$. 
We call an $m$-lattice $\La$  \textit{unimodular} if $\cov(\La)=1$, and we consider the moduli space of unimodular $m$-lattices in $(m+1)$-space with a marked orthogonal vector:
$$X_{m,m+1}:=\{(\Lambda,w):\cov(\La)=1,w\in\mathbb{S}^m,w\perp\La\}.$$
\begin{remark}
    Compared to the space of (non-marked) unimodular $m$-lattices $$L_{m,m+1}:=\{\La:\cov(\La)=1\},$$the space $X_{m,m+1}$ adds only slightly more information -- for each $\La$ there are only two choices of an orthogonal vector $w\in\Sph^m$. All the results below imply directly the analogous results of $L_{m,m+1}$.
\end{remark}
We define a right $\SL(m+1,\R)$ action on $X_{m,m+1}$ by\begin{equation} \label{action of SL on X_m,3} 
   (\La,w). g:=\left(\frac{\La g}{\sqrt{\cov(\La g)}}, \frac{w (^tg^{-1})}{\|w (^tg^{-1})\|} \right),~g\in \SL(m+1,\R),~w\in \mathbb{S}^m, 
\end{equation}where $\La g$ and $w (^tg^{-1})$ are defined using the usual matrix multiplication, and  where $\|\cdot\|$ is the usual Euclidean norm. This action is transitive, and the stabilizer subgroup of the base point $\left(\text{Span}_{\Z}\{e_1,\cdots,e_m\},e_{m+1}\right)$ is $H$ as in \eqref{eq: def of general H} with $\De=\SL(m,
\Z)$. This gives a natural identification $X_{m,m+1}\cong H\backslash G$ and the corresponding action of $\Ga$ on $ H\backslash G$ is the natural one.
For $$x_0:=(\La_0,w_0)\in X_{m,m+1},$$ consider the probability measures $$\mu_{T,x_0}:=\frac{1}{\#\Ga_T}\sum_{\gamma\in \Ga_T}\delta_{x_0.\gamma},~T>0,$$where $\Ga_T$ is as in \eqref{eq:def of Ga T}.
\begin{remark}
    When $m=1$ the space $X_{m,m+1}$ is  identified with $\Sph^1$ the unit circle in $\R^2$. The limiting distribution of $\mu_{T,x_0}$ in the case of $m=1$ was obtained in \cite{Gorodnik2003LatticeAO}, 
\end{remark}
Having the identification $X_{m,m+1}\cong H\backslash G$, we obtain via Theorem \ref{equidistribution result on  G mod H} that the limiting measure of the measures $\mu_{T,x_0}$ as $T\to\infty$ is $\tilde{\nu}_{x_0}$ defined in \eqref{eq:def of nu tilde x0}. In the following we interpret this limiting measure in terms of the $m$-lattices. The projection to the sphere $\mathbb{S}^m$ defined by $$\pi_{\perp}(\La,w):=w,$$endows $X_{m,m+1}$ with a fiber-bundle structure over $\Sph^{m}$, where the fibers are isomorphic to $$X_m:=\SL(m,\Z)\backslash\SL(m,\R).$$ 
The fiber $\pi_{\perp}^{-1}(w)$ may be thought of as the collection of unimodular $m$-lattices orthogonal to $w\in \Sph^{m}$. 
We now give an interpretation for the measures on $\nu_{x_0,w}$ on those fibers. We first consider the function $\Phi_{x_0}$ which was defined in \eqref{ef:def of Phi_x_0}. For an operator T from a hyperplane $U\subset \R^{m+1}$ to another hyperplane $V\subset \R^{m+1}$, we define
\begin{equation*}
    \|\text{T}\|^2_{\text{HS}}:=\sum_{i=1}^m\|\text{T}u_i\|^2,
\end{equation*}
where $\{u_1,u_2,...,u_m\}$ is an orthonormal basis of $U$, and where the norm on the right hand side is the usual Euclidean norm on $\R^{m+1}$. We note that this norm is independent of the choice of an orthonormal basis $\{u_1,u_2,...,u_m\}$.
For an ordered tuple  $B=(u_1,u_2,...,u_m)$ of linearly independent vectors in $\R^{m+1}$ we define the linear map $\text{T}_B:\text{Span}_{\R}  \{e_1,e_2,...,e_m\} \to\text{Span}_{\R}\{u_1,u_2,...,u_m\}$, by sending $e_1\mapsto u_1,...,e_m\mapsto u_m$. Now fix unimodular $m$-lattice $\Lambda_0\subset\R^{m+1}$, and let $\mathscr{B}_0$ be an ordered tuple of linearly independent vectors forming a $\Z$-basis for $\La_0 $. Then for $x_0=(\La_0,w_0)$ the function $\Phi_{x_0}$ in \eqref{ef:def of Phi_x_0}  becomes \begin{equation*}
    \Phi_{x_0} (\La,w):=\sum_{\Span_\Z{\mathscr{B}}=\La}\frac{1}{\|\text{T}_{\mathscr{B}}\circ \text{T}_{\mathscr{B}_0}^{-1}\|_{\text{HS}}^{m^2}}.
\end{equation*}
By identifying $\pi_\perp^{-1}(e_{m+1})$ with $\SL(m,\Z)\backslash \SL(m,\R)$, we obtain the $\SL(m,\Z)\backslash \SL(m,\R)$-invariant measure $\mu_{e_{m+1}}$ supported on $\pi_\perp^{-1}(e_{m+1})$ scaled such that the measure $\nu_{x_0,e_{m+1}}$ defined by $$\nu_{x_0,e_{m+1}}(f):=\int_{\pi^{-1}_\perp(e_{m+1})}f(\La,e_{m+1})\Phi_{x_0}(\La,e_{m+1})d\mu_{e_{m+1}}(\La),~f\in C_c(X_{m,m+1}),$$ is a probability measure. Next, for any other $w\in \Sph^{m}$, we choose $\rho_w\in \SO(m+1,\R)$ such that $w=e_{m+1} \rho_w$, and we define,$$\nu_{x_0,w}:=(\rho_w)_*\nu_{x_0,e_{m+1}},$$ which is the push-forward of the right translation by $\rho_w$ via the right action of $\SO(m+1,\R)$ on $X_{m,m+1}$ defined in \eqref{action of SL on X_m,3}. Note that $\nu_{x_0,w}$ is independent of the choice of $\rho_w$. Finally, we define $\tilde\nu_{x_0}$ by $$\tilde\nu_{x_0}(f)=\int_{\mathbb{S}^m}\nu_{x_0,w}(f)d\mu_{\mathbb{S}^m}(w).$$   
%We note that the above action is transitive.

\begin{theorem}\label{equidistribution result on  G mod H in Xm,m+1}
 Let $\Ga\leq\SL(m+1,\R)$ be a lattice and fix ${x_0}\in X_{m,m+1}$. Then, $\mu_{T,{x_0}}$ converges in the weak-* topology to $\tilde\nu_{x_0}$ as $T\to\infty$. In other words,
 for all $f\in C_c(X_{m,m+1})$, we have 
\begin{equation*}
\lim_{T
\to\infty} \frac{1}{\#\Gamma_{T}}\sum_{\gamma\in \Ga_T}f(x_{0}. \gamma)=\int_{X_{m,m+1}}f(x)d \tilde\nu_{x_{0}}(x).
\end{equation*}

\end{theorem}
\noindent\textbf{An example for $\Ga=\SL(m+1,
\Z)$.} We would like to point out a classical result in number theory that  is related to the orbit $x_0.\Ga$ where $\Ga=\SL(m+1,\Z)$ and $x_0=(\text{Span}_\Z\{e_1,...,e_m\},e_{m+1}) \in X_{m,m+1}$.
The collection of $m$-lattices $$\Z_{m,m+1}:=\left\{\text{Span}_\Z\{e_1,...,e_m\}\gamma:\gamma\in\SL(m+1,\Z)\right\},T>0$$ consists of all primitive (maximal) rank $m$-discrete subgroups of $\Z^{m+1}$. Notice that here the $m$-lattices $\text{Span}_\Z\{e_1,...,e_m\}\gamma$ are without the covolume normalization, and their covolume can be arbitrarily large. It's natural to consider the equivalence relation saying that $m$-lattices are equivalent if they differ by a rotation or by a dilation.  If $\La\leq \R^{m+1}$ is an $m$-lattice, then it's equivelence class $[\La]$ can be naturally represented by a point in $ \text{PGL}(m,\Z)\backslash \text{PGL}(m,\R)/\text{PO}(m,\R)$, where $\text{PGL}$ denotes the projective linear group and where $\text{PO}$ denotes the projective orthogonal group. A particular consequence of a classical result of Schmidt \cite{Schmidt98}, who generalized the works of Roelcke \cite{Roelcke56} and Maass \cite{Maass59}, states that the uniform probability measures on $$\left\{[\La]:\La\in \Z_{m,m+1}\text{ and Cov}(\La)\leq T\right\}$$ converge to the natural measure on $\text{PGL}(m,\Z)\backslash \text{PGL}(m,\R)/\text{PO}(m,\R)$ as $T\to\infty$ 

Using our Theorem \ref{equidistribution result on  G mod H in Xm,m+1}, we can deduce the limiting measure of the following different sequence of probability measures on $\Z_{m,m+1}$ given by $\nu_T:=\frac{1}{\#\Ga_T}\sum_{\gamma\in\Ga_T}\delta_{[\La_0\gamma]},$ where $\La_0=\text{Span}_\Z\{e_1,...,e_m\}.$ To obtain the limiting measure, we note that the equivalence relation modulo rotations and dilations defines a continuous map $[\cdot]:X_{m,m+1}\to \text{PGL}(m,\Z)\backslash \text{PGL}(m,\R)/\text{PO}(m,\R)$. Pushing forward the  measure $\tilde{\nu}_{x_0}$ from Theorem  \ref{equidistribution result on  G mod H in Xm,m+1} via this map gives the limiting measure of $\nu_T$, as $T\to\infty.$

\subsection{Paper organization}
As was mentioned above, to prove our main result we verify the conditions of a theorem due to Gorodnik and Weiss, see Section 2.5 in \cite{Gorodnik2004DistributionOL}. Those conditions are certain volume estimate for the left Haar measure of "skew"-norm balls in $H$, and ergodic theorems of those $H$-balls acting on $G/\Ga$. The organization of the paper is as follows: Section \ref{sec:Estimate of the Haar measure growth of skewed balls in H} states and proves the aforementioned volume estimates, and Section \ref{sec:proof of equid} states and proves the aforementioned ergodic theorems. Finally, Section \ref{sec:finializing the proof using GW} applies the theorem of Gorodnik-Weiss.

\section*{Notations and conventions}
Throughout this paper, for functions $f$ and $g$ we say that as $x\to \al$ where $\al\in \R\cup\{\infty\}$ that $f(x)=O(g(x))$ or $f(x)\ll g(x)$ if there is some $C>0$ such that $|f(x)|\le C |g(x)|$ for all $x$ sufficiently close to $\al$; by $f(x)\asymp g(x)$ we mean $f(x)\ll g(x)$ and $g(x) \ll f(x)$; by $f(x)\sim g(x)$ we mean $\lim_{x\to \al}\left|\frac{f(x)}{g(x)} \right|=L$ for $L\neq0$; by $f(x)=o(g(x))$ we mean $\lim_{x\to\al}f(x)/g(x)=0.$

\section{Volume estimates of expanding skew balls in $H$}\label{sec:Estimate of the Haar measure growth of skewed balls in H}

We begin by describing the left invariant measure on $H$. We consider
\begin{equation}\label{eq:def of the subgroups UADelta}
    U:=\begin{bmatrix}
I_m & 0 \\
\mathbb{\R}^m & 1 
\end{bmatrix}, A=\left \{\begin{bmatrix}
t^{-\frac{1}{m}}I_m & 0 \\
0 & t 
\end{bmatrix}:t>0 \right\},\tilde{\De}:=\begin{bmatrix}
\De & 0 \\
0 &  1 
\end{bmatrix},
\end{equation}
where $\De\leq\SL(m,\R)$ is a lattice, and we denote

\begin{equation}\label{eq:definition of u_v a_t and qtilde}
u_v:=\begin{bmatrix}
I_m & 0 \\
v & 1 
\end{bmatrix},      
a_t:=\begin{bmatrix}
t^{-\frac{1}{m}}I_m & 0 \\
0 & t 
\end{bmatrix}, 
\tilde{q}:=\begin{bmatrix}
q & 0 \\
0 &  1 
\end{bmatrix},  
\end{equation}
for $q\in \Delta,t>0, v \in \R^m$. Any element in $x\in H$ can be represented as $x=u_v a_t\tilde{q} $  uniquely in $v\in\R^{m},t>0,q\in \De$, and we have $$(u_{v_1}a_{t_1}\tilde{q}_1)(u_{v_2}a_{t_2}\tilde{q}_2)=\left(u_{v_1+t_1^{\frac{m+1}{m}}v_2q_1^{-1}}\right)a_{t_1t_2}\left(\tilde{q}_1\tilde{q}_2\right),~v_i\in\R^{m},t_i>0,q_i\in\De.$$
Namely, $H$ is the semi-direct product
$$H\cong \R^{m}\rtimes(\R_{>0} \times \Delta),$$
%Note that clearly (as a product of closed sets), $H$ is a closed subgroup of $G$. Hence it follows from the Theorem 3.58 in \cite{Wa83} that we have the following locally smooth section property:for any $x\in H\backslash G$, there exists a neighborhood $x \in W \subset H\backslash G$ and a smooth map $\sigma:W \to G$ such that 
%\begin{equation*}
%    \pi \circ \sigma = \text{id}.
%\end{equation*}
and it follows that a left Haar measure on $H$ is given by
\begin{align}\label{eq:Haar measure on $H$}
    \int_H f(x)d\mu(x)
    =&\sum_{q\in \Delta}\int_0^{\infty}\int_{\R^m} f\left(u_v a_t\tilde{q}\right) dv \frac{1}{t^{m+2}} dt, 
\end{align}
where the measures $dv$  and $dt$ denote the Lebesgue measures on $\R^m$ and $\R$ correspondingly. 

Throughout this paper, $\|M\|$ for a matrix $M$ denotes it's Hilbert-Schmidt norm.
Following \cite{Gorodnik2004DistributionOL}, for a subgroup $L\leq G$,  $g_1,g_2 \in G$ and $T>0$, we define the so-called ``skewed balls" as follows (defined slightly differently from the definition in \cite{Gorodnik2004DistributionOL}):
\begin{equation*}
    L_T[g_1,g_2]:=\{h\in L: \|g_1^{-1}hg_2\|<T \},
\end{equation*}
and we denote 
\begin{equation}\label{meaning of the subscript T}
    L_T:=L_T[e,e].
\end{equation}
For $g_1, g_2 \in G$, by using Iwasawa decomposition, we denote $k_1,k_2 \in \SO(m+1,\R)$, $\text{g}_1,\text{g}_2\in \GL(m,\R)$, and $\text{v}_1, \text{v}_2\in \R^m$ by
\begin{equation}\label{matrix representation of g1 and g2}
g_1^{-1}=:k_1\begin{bmatrix} \text{g}_1 & 0 \\\text{v}_1 & \det(\text{g}_1)^{-1} \end{bmatrix}, ~g_2=:\begin{bmatrix} \text{g}_2 & 0 \\\text{v}_2 & \det(\text{g}_2)^{-1} \end{bmatrix} k_2.
\end{equation}
We note that the above $k_i,\text{g}_j$, are not unique, but due to $\SO(m+1,\R)$-invariance, all  their appearances below are well defined regardless of their choices.

The key volume estimates we prove in this section are given in the following proposition.
\begin{proposition}\label{computation for H and V} Let $\Gamma(z)$ be the classical Gamma function, and let $$C(m,g_1,g_2):=\frac{m\pi^\frac{m}{2}\Ga(\frac{m^2}{2})}{2\Ga(\frac{m^2}{2}+\frac{m}{2}+1)}\left|\frac{\det(\emph{\text{g}}_1)^m}{ \det(\emph{\text{g}}_2)}\right|.$$There exists $\kappa>0$ such that for a fixed bounded subset $B\subseteq \SL(m+1,\R)$ it holds for all $g_1,g_2\in B$ that    \begin{equation}\label{eq:the estimate for H skew balls}
    \mu(H_T[g_1,g_2])=\left(C(m,g_1,g_2)\sum_{q\in \Delta}\frac{1}{\|\emph{\text{g}}_1q\emph{\text{g}}_2\|^{m^2}}\right)T^{m(m+1)}+O(T^{m(m+1)-\kappa}),  
   \end{equation}
as $T\to\infty$. The implied constant depends on $B$ only.
\end{proposition}
We note that  the sum $\sum_{q\in \Delta}\frac{1}{\|\text{g}_1q\text{g}_2\|^{m^2}}$ is convergent due to Lemma \ref{lemma on the integral and summation in a ball} below.
\begin{remark}
    Statements of this form are known in great generality for a connected $H$, see e.g. Theorems 7.17–7.18 in \cite{GN_book_ergodic_theory}. Here, since our $H$ has infinitely many connected components, the volume of the skew-ball is a sum over the volumes of skew-balls of the connected component, see \eqref{eq:haar measure of skewed H ball as a sum of skewed balls on connected component} below.  Our main challenge here is to quantify the dependence of the error term of the main term appearing of the volume of each skew-ball of the connected component, see Lemma \ref{lem:main estimate for large V skewed balls } below. Proving such a result for more general norms is the main obstacle for  generalizing our results for more general norms.
\end{remark}
As an immediate corollary, we  obtain the following statements which are key requirements for the method of \cite{Gorodnik2004DistributionOL}.
\begin{corollary}
\label{Cor:D1}    \emph{(Uniform volume growth for skewed balls in $H$, property D1 in \cite{Gorodnik2004DistributionOL})}
For any bounded subset $B\subset G$ and any $\e>0$, there are $T_0$ and $\de>0$ such that for all $T>T_0$ and all $g_1,g_2 \in B$ we have:
\begin{equation*}
    \mu \left(H_{(1+\de)T}[g_1,g_2] \right)\le (1+\e)\mu \left(H_T[g_1,g_2] \right).
\end{equation*}
\end{corollary}
Corollary \ref{Cor:D1} follows since Proposition \ref{computation for H and V} states that $$\mu\left(H_T[g_1,g_2]\right)\sim (\text{constant depending on $g_1,g_2$})\times T^{m(m+1)},$$ and importantly, this estimate is \emph{uniform} when $g_1,g_2$ are restricted to a compact set. 
Next we have:
\begin{corollary}[Limit volume ratios, property D2 in \cite{Gorodnik2004DistributionOL}]
For any $g_1,g_2 \in G$. the limit 
\begin{equation}\label{eq:alpha}
    \omega(g_1,g_2):=\lim_{T\to \infty}\frac{\mu \left(H_T[g_1,g_2] \right)}{\mu \left(H_T \right)}=\left|\frac{\det(\emph{\text{g}}_1)^m}{ \det(\emph{\text{g}}_2)}\right|\frac{\sum_{q\in \Delta}\frac{1}{\|\emph{\text{g}}_1q\emph{\text{g}}_2\|^{m^2}}}{\sum_{q\in \Delta}\frac{1}{\|q\|^{m^2}}},
\end{equation}
exists and is positive and finite. 
\end{corollary}
The importance of the explicit expression \eqref{eq:alpha} is that it allows  to deduce, via Thoerem 2.3 of \cite{Gorodnik2004DistributionOL}, the explicit limiting distribution of Theorem \ref{equidistribution result on  G mod H}.

We now detail the structure of the proof of Proposition \ref{computation for H and V}. Consider for $q\in\De$
\begin{equation}\label{definition of V q T}
    V_{q,T}[g_1,g_2]:=\{h\in U A \tilde{q}: \|g_1^{-1}hg_2\|<T \}.
\end{equation}
Since $H=\bigsqcup_{q \in \Delta}UA\tilde{q}$, we have 
\begin{equation}\label{decomposition of H_T into V_T}
    H_T[g_1,g_2]=\bigsqcup_{q \in \Delta}V_{q,T}[g_1,g_2],
\end{equation}
and in particular
\begin{align}\label{eq:haar measure of skewed H ball as a sum of skewed balls on connected component}
    \mu \left(H_T[g_1,g_2] \right)
    =& \sum_{\substack{q\in \Delta}} \mu \left(V_{q,T}[g_1,g_2]\right).
\end{align}
Now denote $V:=UA$,  (the connected component of $H$) and observe that $V_{q,T}[g_1,g_2]$ is the following translated skewed ball of the group $V$:
$$V_{q,T}[g_1,g_2]=V_T[g_1,\tilde{q}g_2]\tilde{q}.$$
More explicitly, note that 
$$V_T[g_1,\tilde{q}g_2]=\{u_va_t:t>0,v\in\R^m\text{ and }\|(g_1^{-1})a_tu_v(\tilde{q}g_2)\|\leq T\},$$
and observe that by  \eqref{eq:Haar measure on $H$},  $$\mu\left(V_{q,T}[g_1,g_2]\right)=\iint\limits_{V_T[g_1,\tilde{q}g_2]}  dv \frac{1}{t^{m+2}} dt.$$When the norm of $q$ is large we have that  $\mu \left(V_{q,T}[g_1,g_2]\right)$, is small (see Lemma \ref{lem:roots lemma}), and the summation over such terms in \eqref{eq:haar measure of skewed H ball as a sum of skewed balls on connected component} doesn't contribute to the main term of \eqref{eq:the estimate for H skew balls}. In fact, we partition the sum into two parts
$$ \left(\sum_{\substack{q\in \Delta,\\ \|\text{g}_1 q \text{g}_2\|< \sqrt{T}}}+\sum_{\substack{q\in \Delta,\\ \sqrt{T}\leq\|\text{g}_1 q \text{g}_2\|}} \right)\mu\left(V_{q,T}[g_1,g_2]\right),$$
and we prove the following lemma.
\begin{lemma} \label{lem:bound on skewed V vol}
      If $g_1,g_2$ vary in a bounded set $B$, then
      \emph{$$\sum_{\substack{q\in \Delta,\\ \sqrt{T}\leq\|\text{g}_1 q \text{g}_2\|}} \mu\left(V_{q,T}[g_1,g_2]\right)=O(T^{m(m+1)-m/2}),$$} where the implied constant depends on $B$ only.
\end{lemma}
For the rest of the terms, we have the following key estimate.
\begin{lemma}\label{lem:main estimate for large V skewed balls }
    There exist $\ka_1,\ka_2>0$ which can be small as we wish and explicitly determined, such that whenever $g_1,g_2$ vary in a bounded set $B$, and $q\in\De$ such that \emph{$\|\text{g}_1q\text{g}_2\|\leq \sqrt{T}$}, it holds that
    \begin{equation*}
       \mu\left(V_{q,T}[g_1,g_2]\right)=\left(C(m,g_1,g_2)\frac{1}{\|\emph{\text{g}}_1q\emph{\text{g}}_2\|^{m^2}}\right)T^{m(m+1)}+O\left(\frac{T^{m(m+1)-\ka_1}}{\|q\|^{m^2}}+\frac{T^{m+\ka_2\frac{m^2}{m+1}}}{\|q\|^{\frac{m^2}{m+1}}}\right), 
    \end{equation*}
    where the implied constant depends on $B$ only.
\end{lemma}

The structure of the rest of the section is as follows. First, in Section  \ref{subsec:proving main prop based on lemmata of V} we prove the main Proposition \ref{computation for H and V} by assuming Lemmata \ref{lem:bound on skewed V vol} - \ref{lem:main estimate for large V skewed balls }. Finally Section \ref{sec:proving the lemmata for the vol estimates}  is dedicaded for the proof of Lemmata \ref{lem:bound on skewed V vol} - \ref{lem:main estimate for large V skewed balls }.
\subsection{Proof of Proposition \ref{computation for H and V} based on Lemmata \ref{lem:bound on skewed V vol} - \ref{lem:main estimate for large V skewed balls }}\label{subsec:proving main prop based on lemmata of V}
The following lemma is essentialy all we need to get Proposition \ref{computation for H and V} from Lemmata \ref{lem:bound on skewed V vol} - \ref{lem:main estimate for large V skewed balls }.
\begin{lemma}\label{lemma on the integral and summation in a ball}
   Let $\De\leq \SL(m,\R)$ by a lattice, and fix an integer $m\geq2$. Then, for $\sigma>-m(m-1)$ it holds that 
   \begin{equation*}
       \sum_{q\in \De_T}\|q\|^\sigma\ll T^{m(m-1)+\sigma},~\text{as }T\to\infty,
   \end{equation*}
   and if $\sigma<-m(m-1)$, 
   the sum above converges with the following tail estimate:\begin{equation*}
       \sum_{q\in \De_T}\|q\|^\sigma=I_\sigma+ O(T^{m(m-1)+\sigma}),~\text{as }T\to\infty.
   \end{equation*}
\end{lemma}
\iffalse
\begin{remark}
Our method of proof doesn't give explicitly the constants $I_{\sigma}$.
\end{remark}
\fi
\begin{proof}
By \cite{DRS93}, there is $C>0$ (an explicit constant), such that
\begin{equation}\label{eq:DRS asymp} 
 \#\De_s=Cs^{m(m-1)}+o(s^{m(m-1)}),~\text{as }s\to\infty.   
\end{equation}
This proves the statement in case that $\sigma=0$. Now, for $\sigma\in\R$
    \begin{align*}
\sum_{q\in \De_T}\|q\|^\sigma 
&= \sum_{q\in \De_T}\int_0^{\|q\|^{\sigma}}1dt\\
&= \sum_{q\in \De_T}\int_0^{\infty}\mathbf{1}_{[t<\|q\|^{\sigma}]}dt\\
&=\int_0^{\infty}\sum_{q\in \De_T}\mathbf{1}_{[t<\|q\|^{\sigma}]}dt.
    \end{align*}
Consider the case that $\sigma>0$. We have
\begin{align*}
\int_0^{\infty}\sum_{q\in \De_T}\mathbf{1}_{[t<\|q\|^{\sigma}]}dt&=\int_0^{T^{\sigma}}\#\{q\in\De:t<\|q\|^\sigma<T^\sigma\}dt \\
&\leq\int_0^{T^{\sigma}}\#\{q\in\De:\|q\|<T\}dt\\
&\underbrace{\ll}_{\eqref{eq:DRS asymp}}T^{m(m-1)+\sigma}.
    \end{align*}  
%Note that in \cite{DRS93}, $s_n$ is normalized to be $1$, but here we use the volume of $\SO(n,\R)$ under the standard spherical parametrization. 

Next, consider $-m(m-1)\ne \sigma<0$. We have
    \begin{align*}
\int_0^{\infty}\sum_{q\in \De_T}\mathbf{1}_{[t<\|q\|^{\sigma}]}dt&=\int_{0}^{\infty}\#\{q\in\De:\|q\|<\min(t^{\frac{1}{\sigma}},T)\} dt\\
&=\int_{T^{\sigma}}^{\infty}\#\{q\in\De:\|q\|<t^{\frac{1}{\sigma}}\}dt+ \int_0^{T^{\sigma}}\#\{q\in\De:\|q\|<T\}dt\\
&\underbrace{=}_{\eqref{eq:DRS asymp}}\int_{T^{\sigma}}^{\infty}\#\{q\in\De:\|q\|<t^{\frac{1}{\sigma}}\}dt+ CT^{m(m-1)+\sigma}+o(T^{m(m-1)+\sigma}),~\text{as }T\to\infty.
    \end{align*}
To prove the claim, it remains to treat the integral $\int_{T^{\sigma}}^{\infty}\#\{q\in\De:\|q\|<t^{\frac{1}{\sigma}}\}dt$. Since $\|g\|\geq1$ for all $g\in\SL(m,\R)$, we have $$\int_{T^{\sigma}}^{\infty}\#\{q\in\De:\|q\|<t^{\frac{1}{\sigma}}\}dt=\int_{T^{\sigma}}^{1}\#\{q\in\De:\|q\|<t^{\frac{1}{\sigma}}\}dt.$$ By \eqref{eq:DRS asymp}, there is $C'>0$ such that $$\#\{q\in\De:\|q\|<t^{\frac{1}{\sigma}}\}\leq C't^{\frac{m(m-1)}{\sigma}},~\text{for all }t\in(0,1).$$
Suppose that $-m(m-1)<\sigma<0$. Then
\begin{align*}
   \int_{T^{\sigma}}^{1}\#\{q\in\De:\|q\|<t^{\frac{1}{\sigma}}\}dt\leq& C'\int_{T^\sigma}^1 t^{\frac{m(m-1)}{\sigma}}dt\\
   &=C'\frac{-1}{\frac{m(m-1)}{\sigma}+1}\left(T^{m(m-1)+\sigma}-1\right)\\
   &\ll T^{m(m-1)+\sigma}.
\end{align*}
It's left to treat the case that $\sigma<-m(m-1)$. Since  $\int_0^1t^\frac{m(m-1)}{\sigma}dt$ converges, we have by dominated convergence that $$\int_0^1\#\{q\in\De:\|q\|<t^{\frac{1}{\sigma}}\}dt=I_{\sigma}$$ for some  constant $=I_{\sigma}$. Finally,
\begin{align*}
  \int_{T^{\sigma}}^{1}\#\{q\in\De:\|q\|<t^{\frac{1}{\sigma}}\}dt=&I_{\sigma}-\int^{T^{\sigma}}_{0}\#\{q\in\De:\|q\|<t^{\frac{1}{\sigma}}\}dt\\
  \leq & I_{\sigma}+C'\left(\int^{T^{\sigma}}_{0}t^\frac{m(m-1)}{\sigma}dt\right)\\
&=I_{\sigma}+O(T^{m(m-1)+\sigma}).
\end{align*}

\end{proof}
For the following, we note the basic fact -- when $g_1,g_2$ vary in a bounded set $B$, there exists constants $C^+,C^->0$ such that
\begin{equation}\label{eq:lower bound on A_m}
C^-\|q\|\leq \|\text{g}_1 q \text{g}_2\|\leq C^+\|q\|,
\end{equation}
where $\text{g}_1,\text{g}_2$ are given by \eqref{matrix representation of g1 and g2}.
\begin{proof}[Proof of Proposition \ref{computation for H and V}]
We assume that $g_1,g_2$ vary in a fixed bounded set $B$. We have:  
\begin{align}
    \mu \left(H_T[g_1,g_2] \right)
    =& \sum_{q\in \Delta} \mu \left(V_{q,T}[g_1,g_2]\right)\nonumber\\
    \underbrace{=}_{\text{Lemma \ref{lem:bound on skewed V vol} }}&~\sum_{\substack{q\in \Delta,\nonumber\\ \|\text{g}_1q\text{g}_2\|< \sqrt{T}}}\mu \left(V_{q,T}[g_1,g_2]\right)+O(T^{m(m+1)-m/2})\nonumber\\
\underbrace{=}_{\text{Lemma \ref{lem:main estimate for large V skewed balls } }}&\left(C(m,g_1,g_2)\sum_{\substack{q\in \Delta,\\ \|\text{g}_1 q \text{g}_2\|<\sqrt{T}}}\frac{1}{\|\text{g}_1q\text{g}_2\|^{m^2}}\right)T^{m(m+1)}\nonumber\\
&+O\left(\sum_{\substack{q\in \Delta,\nonumber\\ \| q\|<\sqrt{T}/C^-}}\frac{T^{m(m+1)-\ka_1}}{\|q\|^{m^2}}+\frac{T^{m+\ka_2\frac{m^2}{m+1}}}{\|q\|^{\frac{m^2}{m+1}}}\right).
\end{align} 
We now replace the summation $\sum_{\substack{q\in \Delta,\\ \|\text{g}_1 q \text{g}_2\|<\sqrt{T}}}$ in the last expression by $\sum_{q\in \Delta}$:
\begin{align}
   \left(C(m,g_1,g_2)\sum_{\substack{q\in \Delta,\\ \|\text{g}_1 q \text{g}_2\|<\sqrt{T}}}\frac{1}{\|\text{g}_1q\text{g}_2\|^{m^2}}\right)T^{m(m+1)}&\nonumber\\
   =\left(C(m,g_1,g_2)\sum_{q\in \Delta}\frac{1}{\|\text{g}_1q\text{g}_2\|^{m^2}}\right)T^{m(m+1)}+O&\left(\sum_{\substack{q\in \Delta,\\ \| q \|\geq \sqrt{T}/C^+}}\frac{1}{\|q\|^{m^2}}\right)T^{m(m+1)}\nonumber
   \end{align}
   \begin{align}
   \underbrace{=}_{\text{Lemma } \ref{lemma on the integral and summation in a ball}}\left(C(m,g_1,g_2)\sum_{q\in \Delta}\frac{1}{\|\text{g}_1q\text{g}_2\|^{m^2}}\right)T^{m(m+1)}&+O\left(T^{m(m+1)-m/2}\right).\nonumber
\end{align}
Finally, we estimate the terms in the big-O notation above by again applying Lemma \ref{lemma on the integral and summation in a ball}: 
\begin{align}
 \sum_{\substack{q\in \Delta,\nonumber\\ \| q\|<\sqrt{T}/C^-}}\frac{T^{m(m+1)-\ka_1}}{\|q\|^{m^2}}=O(T^{m(m+1)-\ka_1}),   
\end{align}
and
\begin{align}
 \sum_{\substack{q\in \Delta,\nonumber\\ \| q\|<\sqrt{T}/C^-}} \frac{T^{m+\ka_2\frac{m^2}{m+1}}}{\|q\|^{\frac{m^2}{m+1}}}=O(T^{0.5m(m+1)-\frac{(0.5-\ka_2)m^2}{m+1}}).  
\end{align}
The proof is complete by combining all the above estimates.
\end{proof}
\subsection{Proving Lemmata \ref{lem:bound on skewed V vol} - \ref{lem:main estimate for large V skewed balls }}\label{sec:proving the lemmata for the vol estimates}
We recall \eqref{matrix representation of g1 and g2}:
\begin{equation*}
    g_1^{-1}=:k_1\begin{bmatrix} \text{g}_1 & 0 \\\text{v}_1 & \det(\text{g}_1)^{-1} \end{bmatrix}, ~g_2=:\begin{bmatrix} \text{g}_2 & 0 \\\text{v}_2 & \det(\text{g}_2)^{-1} \end{bmatrix} k_2.
\end{equation*}
By bi-$\SO(m+1,\R)$ invariance, we have for $q\in\De$
\begin{align*}
  V_T[g_1,\tilde{q}g_2]&=\left\{u_va_t:v\in\R^m,t>0,\left\|g_1^{-1}u_va_t(\tilde{q}g_2) \right\|\leq T\right\}\\ &=\left\{u_va_t:v\in\R^m,t>0,\left\|\begin{bmatrix} \text{g}_1 & 0 \\\text{v}_1 & \det(\text{g}_1)^{-1} \end{bmatrix}u_va_t\begin{bmatrix}
q & 0 \\
0 &  1 
\end{bmatrix} \begin{bmatrix} \text{g}_2 & 0 \\\text{v}_2 & \det(\text{g}_2)^{-1} \end{bmatrix}\right\|\leq T\right\},
\end{align*}
where $u_v\text{ and }a_t$ are given by \eqref{eq:definition of u_v a_t and qtilde}. We will now describe the above skew-ball as a family of ellipsoids. 
A computation shows

\begin{align*}
   \begin{bmatrix} \text{g}_1 & 0 \\\text{v}_1 & \det(\text{g}_1)^{-1} \end{bmatrix}u_va_t\begin{bmatrix}
q & 0 \\
0 &  1 
\end{bmatrix}& \begin{bmatrix} \text{g}_2 & 0 \\\text{v}_2 & \det(\text{g}_2)^{-1} \end{bmatrix}\\
=&\begin{bmatrix}
t^{-\frac{1}{m}}\text{g}_1q \text{g}_2 & 0 \\
t^{-\frac{1}{m}}\text{v}_1q \text{g}_2+\det(\text{g}_1)^{-1}(t^{-\frac{1}{m}}vq \text{g}_2 +  t\text{v}_2) &  \det(\text{g}_1\text{g}_2)^{-1}t 
\end{bmatrix}.
\end{align*}
Upon taking the sum of squares, we conclude that $\|g_1^{-1}u_va_t(\tilde{q}g_2)\| \le T$ is equivalent to
\begin{equation}\label{equation defining the ball of integration}
    \|\text{v}_1q \text{g}_2+\det(\text{g}_1)^{-1} vq \text{g}_2+\det(\text{g}_1)^{-1} t^{\frac{1}{m}+1}\text{v}_2\|^2 
    \le -\det(\text{g}_1\text{g}_2)^{-2} t^{\frac{2}{m}+2}+t^{\frac{2}{m}}T^2-\|\text{g}_1 q \text{g}_2\|^2.
\end{equation}
Let 
\begin{equation}\label{eq:def of R_q,T}
    R_{q,T}(t):=-\det(\text{g}_1\text{g}_2)^{-2} t^{\frac{2}{m}+2}+t^{\frac{2}{m}}T^2-\|\text{g}_1 q \text{g}_2\|^2,
\end{equation}
and note that the solutions in $v\in\R^m$ for  \eqref{equation defining the ball of integration}  define an interior of an ellipsoid 
\begin{equation}\label{eq:def of D}
    D_{q,T,t}:=\{v\in\R^m:\|\text{v}_1q \text{g}_2+\det(\text{g}_1)^{-1} vq \text{g}_2+\det(\text{g}_1)^{-1} \text{v}_2t^{\frac{1}{m}+1}\|^2 \leq R_{q,T}(t)\}.
\end{equation}We have
\begin{equation}\label{eq:volume of D_gamma,t}
    \vol(D_{q,T,t})= \begin{cases}
        v_{m}\left|\frac{\det(\text{g}_1)^m}{\det(\text{g}_2)}\right|\left(R_{q,T}(t) \right)^{\frac{m}{2}}& \text{if } R_{q,T}(t)>0, \\
        0 & \text{otherwise.}
    \end{cases}  
\end{equation}
where $\vol(D_{q,T,t})$ is the Lebesgue measure of the ellipse $D_{q,T,t}$. Here $v_m=\frac{\pi^{\frac{m}{2}}}{\Ga(\frac{m}{2}+1)}$ denotes the volume of Euclidean ball of radius one in $\R^m$.
We conclude that
\begin{equation}\label{eq:the measure of V skew ball}
    \mu \left(V_{q,T}[g_1,g_2]\right):=v_{m}\left|\frac{\det(\text{g}_1)^m}{\det(\text{g}_2)}\right|\int_{\{t>0:R_{q,T}(t)>0\}} \frac{\left(R_{q,T}(t) \right)^{\frac{m}{2}}}{t^{m+2}}dt.
\end{equation}

As we will show now, the set $\{t>0:R_{q,T}(t)>0\}$ is either an interval $(\al_{q,T},\be_{q,T}),$ or  empty when the norm of $q$ is large. In particular, we will give estimates of the roots $\al_{q,T},\be_{q,T}$ which are essential in our proof.
 %where $C_{\alpha}(g_1,g_2)\in [1,2]$ (depending on how much overlap two balls may have) is a constant related to $g_1,g_2$. Observe that when $\text{v}_2=0$, we have two ellipses merge into one, and $C_{\alpha}(g_1,g_2)=0$.

%Note if $g_1=g_2=e$, then $\text{g}_1=\text{g}_2=I_2$, $\text{v}_1=\text{v}_2=0$, $\det(\text{g}_1)^{-1}=\det(\text{g}_2)^{-1}=1$ and the equation above reduces to
%\begin{equation}
%    \|vq\|^2\le -t^{\frac{2}{m}+2}+t^{\frac{2}{m}}T^2-\|q\|^2.
%\end{equation}
%In this case, $C_{\alpha}(e,e)=1$.

\begin{lemma}\label{lem:roots lemma}
Consider for $T>0$,\emph{\begin{equation}\label{eq:maximal value}
     M_{T}:=\frac{m}{m+1}\left(\frac{1}{m+1}\right)^{\frac{1}{m}} \frac{T^{2+\frac{2}{m}}}{|\det(\text{g}_1\text{g}_2)^{-1}|^{\frac{2}{m}}},
 \end{equation}} and
 \emph{\begin{equation}\label{eq:critical value where M_T is attained}
     \theta_T:=\frac{1}{\sqrt{1+m}} \frac{T}{|\det(\text{g}_1\text{g}_2)^{-1}|}.
 \end{equation}}Then:
\begin{enumerate}
    \item $\{t>0:R_{q,T}(t)>0\}\neq\emptyset$ if and only if \emph{$\|\text{g}_1q\text{g}_2\|^2< M_T.$}
\item Suppose that $\{t>0:R_{q,T}(t)>0\}\neq\emptyset$. Then $\{t>0:R_{q,T}(t)>0\}=(\al_{q,T},\be_{q,T})$, and the following estimates hold:
\begin{equation}\label{bounds on roots}
   0<\al_{q,T}<\theta_T<\be_{q,T}<\sqrt{m+1} \theta_T, 
 \end{equation}
 
 \emph{\begin{equation}\label{range for a}
    \al_{q,T}^{\frac{2}{m}}\in
    \left(\frac{\|\text{g}_1 q \text{g}_2\|^2}{T^2}, \frac{(m+1)\|\text{g}_1 q \text{g}_2\|^2}{mT^2}\right),
\end{equation} }
and moreover, if $g_1,g_2$ vary in a bounded set $B$, and $q\in\De$ such that \emph{$\|\text{g}_1q\text{g}_2\|\leq \sqrt{T}$}
\emph{\begin{equation}\label{eq:asymp for a when A small then sqrt T}
    \al_{q,T}^{\frac{2}{m}}=\frac{\|\text{g}_1 q \text{g}_2\|^2}{T^2}+O\left(\frac{1}{T^{m+3}}\right).
\end{equation}}
    \end{enumerate}

\end{lemma}
 \begin{proof}
   First, observe that the maximal value of $-\det(\text{g}_1\text{g}_2)^{-2} t^{\frac{2}{m}+2}+t^{\frac{2}{m}}T^2$ for $t>0$ is \eqref{eq:maximal value}. Since $R_{q,T}(t)=-\det(\text{g}_1\text{g}_2)^{-2} t^{\frac{2}{m}+2}+t^{\frac{2}{m}}T^2-\|\text{g}_1 q \text{g}_2\|^2$, we conclude that $\{t>0:R_{q,T}(t)>0\}\neq\emptyset$ if only if $\|\text{g}_1q\text{g}_2\|^2< M_T.$

Next, observe that \eqref{eq:critical value where M_T is attained} is the  unique positive critical point of $R_{q,T}(\cdot)$. More precisely, $R_{q,T}(\cdot)$  is monotonic increasing in $(0,\theta_T)$ and monotonic decreasing in $(\theta_T,\infty)$. Thus, $R_{q,T}(\cdot)$ has two positive roots $\al_{q,T},\be_{q,T}$,and the bounds \eqref{bounds on roots} are concluded  by noting that $R_{q,T}(0)<0$ and that $R_{q,T}(\sqrt{m+1}\theta_T)<0$.

It remains to prove the more precise bounds for the root $\al_{q,T}$. By rearranging terms in $$-\det(\text{g}_1\text{g}_2)^{-2} \al_{q,T}^{\frac{2}{m}+2}+\al_{q,T}^{\frac{2}{m}}T^2-\|\text{g}_1 q \text{g}_2\|^2=0,$$  we get that
\begin{equation}\label{eq:recursive form of a}
   \al_{q,T}^{\frac{2}{m}}=\frac{\|\text{g}_1 q \text{g}_2\|^2}{T^2-\det(\text{g}_1\text{g}_2)^{-2} \al_{q,T}^2}, 
\end{equation}
and by using that $\al_{q,T}<\theta_T$, see \eqref{bounds on roots}, we conclude that
\begin{equation*}
    \al_{q,T}^{\frac{2}{m}}=\frac{\|\text{g}_1 q \text{g}_2\|^2}{T^2-\det(\text{g}_1\text{g}_2)^{-2} \al_{q,T}^2}\in
    \left(\frac{\|\text{g}_1 q \text{g}_2\|^2}{T^2}, \frac{(m+1)\|\text{g}_1 q \text{g}_2\|^2}{mT^2}\right).
\end{equation*} 
In particular, it follows that
\begin{equation}
        \al_{q,T}\asymp \frac{\|\text{g}_1 q \text{g}_2\|^m}{T^m}. \label{estimate for a in all range}
\end{equation}

Finally, assume that $\|\text{g}_1 q \text{g}_2\|\leq\sqrt{T}$ and that $g_1,g_2$ vary in a bounded set. Then 
\begin{align*}
   \al_{q,T}^{\frac{2}{m}}=&\frac{\|\text{g}_1q\text{g}_2\|^2}{T^2}\frac{1}{1-\frac{\det(\text{g}_1\text{g}_2)^{-2}\al_{q,T}^2}{T^2}}\\
   =&\frac{\|\text{g}_1 q \text{g}_2\|^2}{T^2}\left(1+\frac{\det(\text{g}_1\text{g}_2)^{-2}\al_{q,T}^2}{T^2}+\cdots \right)\\
   =&\frac{\|\text{g}_1 q \text{g}_2\|^2}{T^2}+O\left(\frac{1}{T^{m+3}}\right).\tag{using \eqref{range for a} with $\|\text{g}_1 q \text{g}_2\|\leq\sqrt{T}$}
\end{align*}where in the last line we also used the fact that for $g_1,g_2$ that vary in a bounded set it holds that $\det(\text{g}_1\text{g}_2)^{-2}$ is bounded above (here $\text{g}_i$ are given by \eqref{matrix representation of g1 and g2}).
 
 \end{proof}
 We now conclude that whenever $\|\text{g}_1q\text{g}_2\|^2\leq M_T,$ we have  by \eqref{eq:volume of D_gamma,t} that
 \begin{equation}\label{eq:measure of V skewed balls as integral}
  \mu \left(V_{q,T}[g_1,g_2] \right)
   = \frac{\pi^{\frac{m}{2}}}{\Ga(\frac{m}{2}+1)}\left|\frac{\det(\text{g}_1)^m}{\det(\text{g}_2)}\right|\int_{\al_{q,T}}^{\be_{q,T}} \frac{(R_{q,T}(t))^{\frac{m}{2}}}{t^{m+2}}dt. 
\end{equation}
\subsubsection{Proving Lemma \ref{lem:bound on skewed V vol}}
The proof of Lemma \ref{lem:bound on skewed V vol} consists of the following bound and on the summation Lemma \ref{lemma on the integral and summation in a ball}.

\begin{lemma}\label{lem:estimate of V skewed balls as power of T and fract of A_m to power m^2}It holds that
\emph{\begin{align}
    \mu(V_{q,T}[g_1,g_2])
    \le & v_m\left|\frac{\det(\text{g}_1)^m}{\det(\text{g}_2)}\right|\frac{T^{m(m+1)}}{\|\text{g}_1 q \text{g}_2\|^{m^2}}\label{eq:bound on measure of Vga T} \\
    \leq&v_m\left|\frac{\det(\text{g}_1)^m}{\det(\text{g}_2)}\right| \|\text{g}_1^{-1}\|^{m^2}\|\text{g}_2^{-1}\|^{m^2}\frac{T^{m(m+1)}}{\|q\|^{m^2}}.\label{eq:bound on volume of connected component}
\end{align}}
     In particular, when $g_1,g_2$ vary in a compact subset of $\SL(m+1,\R)$, it holds that
    $$\mu(V_{q,T}[g_1,g_2])\leq T^{m(m+1)}O\left(\frac{1}{\|q\|^{m^2}}\right).$$
\end{lemma}
\begin{proof}
   We have
   \begin{align*}
      &\int_{\al_{q,T}}^{\be_{q,T}} \frac{(-\det(\text{g}_1\text{g}_2)^{-2}t^{\frac{2}{m}+2}+t^{\frac{2}{m}}T^2-\|\text{g}_1 q \text{g}_2\|^2)^{\frac{m}{2}}}{t^{m+2}}dt\\
    \le &\int_{\al_{q,T}}^{\be_{q,T}} \frac{(t^{\frac{2}{m}}T^2)^{\frac{m}{2}}}{t^{m+2}}dt\\
    =& T^m\int_{\al_{q,T}}^{\be_{q,T}}\frac{1}{t^{m+1}}\\ 
    \leq& T^m\frac{1}{m\al_{q,T}^m}.
   \end{align*}
Using \eqref{range for a},  
$$\frac{1}{\al_{q,T}^m}\leq\frac{T^{m^2}}{\|\text{g}_1 q \text{g}_2\|^{m^2}}\leq T^{m^2}\frac{\|\text{g}_1^{-1}\|^{m^2}\|\text{g}_2^{-1}\|^{m^2}}{\|q\|^{m^2}},$$
which proves the claim upon recalling \eqref{eq:measure of V skewed balls as integral}. 
\end{proof}
\begin{proof}[Proof of Lemma \ref{lem:bound on skewed V vol}]
In the following  $g_1,g_2$ vary in a bounded set. By Lemma \ref{lem:estimate of V skewed balls as power of T and fract of A_m to power m^2}, 
 \begin{align*}
\sum_{\substack{q\in \Delta,\\ \sqrt{T}<\|\text{g}_1 q \text{g}_2\|}}\mu(V_{q,T}[g_1,g_2])\le   T^{m(m+1)}O\left(\sum_{\substack{q\in \Delta,\\ \sqrt{T}<\|\text{g}_1 q \text{g}_2\|}} \frac{1}{\|q\|^{m^2}}\right).
\end{align*} 
The proof is complete by recalling \eqref{eq:lower bound on A_m} and by applying Lemma \ref{lemma on the integral and summation in a ball}.
\end{proof}

\subsubsection{Proving Lemma \ref{lem:main estimate for large V skewed balls }}
\iffalse
Meanwhile by Taylor expansion,
\begin{align}
    a=&\frac{\|\text{g}_1 q \text{g}_2\|^m}{T^m}\left(\frac{1}{1-\frac{\det(\text{g}_1)^{-1}^2\det(\text{g}_2)^{-1}^2 a^2}{T^2}}\right)^{\frac{m}{2}} \nonumber\\ 
    =&\frac{\|\text{g}_1 q \text{g}_2\|^m}{T^m}\left(1+\frac{\det(\text{g}_1\text{g}_2)^{-2} a^2}{T^2}+\cdots \right)^{\frac{m}{2}} \label{taylor expansion for a}
\end{align}

Moreover, 
%for the $q$ with smaller norms, 
when $\|\text{g}_1 q \text{g}_2\|\leq T$ we have that it follows from \eqref{taylor expansion for a} that
\begin{equation}\label{estimate for a in smaller range taylor expansion}
    a=\frac{\|\text{g}_1 q \text{g}_2\|^m}{T^m}\left(1+\det(\text{g}_1\text{g}_2)^{-2} O_{g_1,g_2}\left(\frac{1}{T^2} \right)\right)^{\frac{m}{2}}, \text{ whenever }\|\text{g}_1 q \text{g}_2\|\le T.
\end{equation}
\fi
In this section we assume throughout that $\|\text{g}_1q\text{g}_2\|\leq\sqrt{T}$. For $\e_1\in(0,1)$ we define $\de=\de(T,q,\e_1)$ by
\begin{equation}\label{eq:definition of delta}
  \de:=\frac{\|\text{g}_1 q \text{g}_2\|}{T^{1+\e_1}},\end{equation} we define $\al_\de=\al_\de(q,T,\e_1)$ by\begin{equation}\label{eq:defintion of a_delta}
    \al_\de^{\frac{1}{m}}:= \al_{q,T}^{\frac{1}{m}}+\de,
\end{equation}and for $\e_2\in(0,1)$ we let $\la=\la(q,T,\e_2)$ 
\begin{equation}\label{eq:definition of c}
    \la:=\left(\frac{\|\text{g}_1 q \text{g}_2\|}{T^{\e_2}}\right)^{\frac{m}{m+1}}.
\end{equation}By the bounds \eqref{bounds on roots} and \eqref{range for a}, it follows that when  $\|\text{g}_1 q\text{g}_2\|\leq\sqrt{T}$ and $g_1,g_2$ vary in a compact set, there exists a uniform $T_0$ such that for all $T\geq T_0,$ $$\al<\al_{\de}< \lambda<\be.$$In what follows the reader should keep in mind that $T\geq T_0$. We consider the following partition of the integral appearing in  \eqref{eq:measure of V skewed balls as integral},
\begin{align}
&\int_{\al_{q,T}}^{\be_{q,T}} \frac{(R_{q,T}(t))^{\frac{m}{2}}}{t^{m+2}}dt
=\left(\int_{\al_{q,T}}^{\al_{\de}(q,T,\e_1)}+\int_{\al_{\de}(q,T,\e_1)}^{\la(q,T,\e_2)}+\int_{\la(q,T,\e_2)}^{\be_{q,T}}\right) \frac{(R_{q,T}(t))^{\frac{m}{2}}}{t^{m+2}}dt\label{eq:partition of integral into four part}
\end{align}
In the following Lemmata we will estimate each of those integrals, and Lemma \ref{lem:main estimate for large V skewed balls } will follow by combining those estimates. The main term will come from  the integral in the range $\int_{\al_{\de}(q,T,\e_1)}^{\la(q,T,\e_2)}$. 
\begin{lemma} \label{lemma for integral from a to a delta}
    Suppose that \emph{$\|\text{g}_1 q \text{g}_2\|\leq\sqrt{T}$} and that $g_1,g_2$ vary in a bounded set of $\SL(m+1,\R)$. Fix $\e_1\in(0,1)$. Then
    \begin{align}\label{eq:first integral}
        \int_{\al_{q,T}}^{\al_{\de}(q,T,\e_1)} \frac{(R_{q,T}(t))^{\frac{m}{2}}}{t^{m+2}}dt =O\left(\frac{T^{m(m+1)-\e_1\frac{m+2}{2}}}{\|q\|^{m^2}}\right)   
    \end{align}
\end{lemma}    
\begin{proof}
Note since $\al_{q,T}$ is a root, we have that $R_{q,T}(t)$ is small in the range of the integral \eqref{eq:first integral}. In fact, we will show that it's small enough compared with $t^{-m+2}$ which is large in that range. In the following, to ease the reading, we will omit from the notations the dependencies on $q$, $T$ and $\e_1$. We substitute the variable $s=t^{\frac{1}{m}}$ in \eqref{eq:def of R_q,T}:
\begin{align*}
    f(s):=R_{q,T}(s)=-\det(\text{g}_1\text{g}_2)^{-2} s^{2m+2}+T^2s^2-\|\text{g}_1 q \text{g}_2\|^2.
\end{align*}
We apply Taylor expansion to $f(s)$ in the range $s\in[\al^{\frac{1}{m}},\al_\de^{\frac{1}{m}}]$
at $s=\al^{\frac{1}{m}}$ which gives
\begin{align}
    f(\al_\de^\frac{1}{m})&=f(\al^{\frac{1}{m}}+\de)=\underbrace{f(\al^{\frac{1}{m}})}_{=0}+f'(\al^{\frac{1}{m}})\de+\frac{f''(\xi)}{2}\de^2\nonumber\\&= [-(2m+2)\det(\text{g}_1\text{g}_2)^{-2} \al^\frac{2m+1}{m}+2T^2\al^\frac{1}{m}]\de+\frac{f''(\xi)}{2}\de^2\nonumber\\
    &\ll T^2\al^\frac{1}{m}\de+f''(\xi)\de^2\label{eq:bound by taylor}
\end{align}
where $\xi\in [\al^\frac{1}{m},\al_\de^{\frac{1}{m}}]$. In the last line we used that for bounded $g_1,g_2$, $|\det(\text{g}_1\text{g}_2)^{-1}|$ is bounded above and away from zero. We not bound the second derivative $f''(\xi)$ over $[\al^\frac{1}{m},\al_\de^{\frac{1}{m}}]$:
\begin{align}
    |f''(\xi)|
    =&|-(2m+2)(2m+1)\det(\text{g}_1\text{g}_2)^{-2} \xi^{2m}+2T^2|\nonumber\\
    \le & (2m+2)(2m+1)\det(\text{g}_1\text{g}_2)^{-2} \al_\de^\frac{2m}{m}+2T^2\nonumber \\
    \ll & T^2.
\label{eq:bound on sectond der}\end{align}Thus by \eqref{eq:bound by taylor} and \eqref{eq:bound on sectond der},
\begin{equation}
f(\al_\de^{\frac{1}{m}})\ll T^2\al_\de^\frac{1}{m}\de.\label{eq:bound for f(al_de 1 ove m)}
\end{equation}We note that for all large $T$ (independently of $q$ and $\e_1$) $f$ is monotonically increasing on the interval $(\al^\frac{1}{m},\al_\de^{\frac{1}{m}})$. To see this, recall that $\theta_T^\frac{1}{m}\sim T^\frac{1}{m}$ is a critical point for $f(s)$,  $f(s)$ is monotonically increasing in $(0,\theta_T^\frac{1}{m})$ (see Lemma \ref{lem:roots lemma} and it's proof), and $\al_\de=o(1)$. Then,  
\begin{align*}
    \int_{\al}^{\al_{\de}} \frac{(R_{q,T}(t))^{\frac{m}{2}}}{t^{m+2}}dt
    \le&\frac{1}{\al^{m+2}}\int_{\al}^{\al_{\de}} (R_{q,T}(t))^{\frac{m}{2}}dt\\
    =&\frac{1}{\al^{m+2}}\int_{\al^\frac{1}{m}}^{\al_{\de}^\frac{1}{m}} f(s)^{\frac{m}{2}}ms^{m-1}ds
      \end{align*}
   \begin{align*}
    \ll&\frac{1}{\al^{m+2}}(f(\al_\de^{\frac{1}{m}})^\frac{m}{2}\al_\de^{\frac{m-1}{m}})\de\tag{monotonicity of $f(s)$}\\
    \ll&\frac{1}{\al^{m+1+\frac{1}{m}}}f(\al_\de^{\frac{1}{m}})^\frac{m}{2}\de\tag{$\al_\de\ll\al$}
    \\
    \ll&\frac{\left\{T^2\al_\de^{\frac{1}{m}}\right\}^{\frac{m}{2}}}{\al^{m+1+\frac{1}{m}}}\cdot\delta^{\frac{m+2}{2}}\tag{by \eqref{eq:bound for f(al_de 1 ove m)}}\\
     \ll& \frac{\left\{T^2\al^{\frac{1}{m}}\right\}^{\frac{m}{2}}}{\al^{m+1+\frac{1}{m}}}\cdot\delta^{\frac{m+2}{2}} \tag{$\al_\de\ll\al$}
    \end{align*}
Using the bounds of $\al$ in \eqref{range for a} and by applying the definition of $\delta$ in \eqref{eq:definition of delta}, we get\begin{equation*}
\frac{\left\{T^2\al^{\frac{1}{m}}\right\}^{\frac{m}{2}}}{\al^{m+1+\frac{1}{m}}}\cdot\delta^{\frac{m+2}{2}}\ll\frac{T^{m(m+1)-\e_1\frac{m+2}{2}}}{\|\text{g}_1 q \text{g}_2\|^{m^2}},  
\end{equation*}
which concludes our proof.    
\end{proof}

We proceed to treat the middle integral, which will be the main term in \eqref{eq:partition of integral into four part}. 
\begin{lemma}\label{lem:integral from a_delta to c}
Let \emph{$\|\text{g}_1 q \text{g}_2\|\leq\sqrt{T}$}, and assume that $g_1,g_2$ vary in a bounded set of $\SL(m+1,\R)$. Suppose that $0<\e_1<2\e_2<1$, and let \begin{equation*}
\e:=\min\{\e_1,2\e_2-\e_1,1+\frac{1-2\e_2}{m+1}\}.
\end{equation*}Then \emph{\begin{align}
   &\int_{\al_{\de}(q,T,\e_1)}^{\la(q,T,\e_2)}\frac{(R_{q,T}(t))^{\frac{m}{2}}}{t^{m+2}}dt
=\frac{m\Ga(\frac{m}{2}+1)\Ga(\frac{m^2}{2})}{2\Ga(\frac{m^2}{2}+\frac{m}{2}+1)}
\cdot \frac{T^{m(m+1)}}{\|\text{g}_1 q \text{g}_2\|^{m^2}}+O\left(\frac{T^{m(m+1)-\e}}{\|q\|^{m^2}}\right). \label{eq:main term estimate for integral in a delta to c}
\end{align}}
\end{lemma}
\begin{proof}
%Suppose that $\|\text{g}_1 q \text{g}_2\|\leq\sqrt{T}$, and assume that $g_1,g_2$ vary in a bounded set. In particular, recall that $|\det(\text{g}_1\text{g}_2)^{-1}|$ will be bounded above. 
To ease the reading, we will omit from some notations the dependencies on $q$, $T$ and $\e_1,\e_2$. 
We rewrite our integral \eqref{eq:main term estimate for integral in a delta to c} as 
\begin{align*}
\int_{\al_{\de}}^{\la}\frac{(-\det(\text{g}_1\text{g}_2)^{-2} t^{\frac{2}{m}+2}+t^{\frac{2}{m}}T^2-\|\text{g}_1 q \text{g}_2\|^2)^{\frac{m}{2}}}{t^{m+2}}dt&\\
=\int_{\al_\de}^\la\frac{(t^{\frac{2}{m}}T^2-\|\text{g}_1 q \text{g}_2\|^2)^{\frac{m}{2}}}{t^{m+2}}
& \left(1-\frac{\det(\text{g}_1\text{g}_2)^{-2}t^{\frac{2}{m}+2}}{t^{\frac{2}{m}}T^2-\|\text{g}_1 q \text{g}_2\|^2} \right)^{\frac{m}{2}}dt.    
\end{align*}
We now show that in the range $t\in(\al_\de,\la)$ it holds that 
\begin{equation}\label{much less than 1}
   \frac{\det(\text{g}_1\text{g}_2)^{-2}t^{\frac{2}{m}+2}}{t^{\frac{2}{m}}T^2-\|\text{g}_1 q \text{g}_2\|^2}= O\left(\frac{1}{T^{2\e_2-\e_1}}\right).
\end{equation}
Recall that we assume here $\|\text{g}_1 q \text{g}_2\|\leq\sqrt{T}$, and as $g_1,g_2$ vary in a bounded set, $|\det(\text{g}_1\text{g}_2)^{-1}|$ will be bounded above. We have for $t\in(\al_\de,\la)$
\begin{align*}
    \frac{\det(\text{g}_1\text{g}_2)^{-2}t^{\frac{2}{m}+2}}{t^{\frac{2}{m}}T^2-\|\text{g}_1 q \text{g}_2\|^2}
    \ll &  \frac{\la^{2\frac{m+1}{m}}}{(\al^\frac{1}{m}+\de)^2T^2-\|\text{g}_1 g \text{g}_2\|^2}  \\
    \ll & \frac{\la^{2\frac{m+1}{m}}}{O(T^{-m-3})+2\al^\frac{1}{m}\de T^2+\de^2 T^2} \tag{by \eqref{eq:asymp for a when A small then sqrt T}}\\
    \ll &\frac{\la^{2\frac{m+1}{m}}}{\al^\frac{1}{m}\de T^2} \\
    \leq & \frac{(\nicefrac{\|\text{g}_1 q \text{g}_2\|}{T^{\e_2}})^2}{ \frac{\|\text{g}_1 q \text{g}_2\|}{T}\frac{\|\text{g}_1 q \text{g}_2\|}{T^{1+\e_1}}T^2}=O\left(\frac{1}{T^{2\e_2-\e_1}}\right)\tag{using \eqref{range for a} for bounding $\al$}.
\end{align*} 

Now in view of the estimate $(1+x)^{\alpha}\sim 1+\alpha x$ as $x\to0$, we conclude that 
\begin{align}
 \int_{\al_\de}^\la \frac{(t^{\frac{2}{m}}T^2-\|\text{g}_1 q \text{g}_2\|^2)^{\frac{m}{2}}}{t^{m+2}}
\left(1-\frac{\det(\text{g}_1\text{g}_2)^{-2}t^{\frac{2}{m}+2}}{t^{\frac{2}{m}}T^2-\|\text{g}_1 q \text{g}_2\|^2} \right)^{\frac{m}{2}}dt\nonumber\\=\left(\int_{\al_\de}^\la \frac{(t^{\frac{2}{m}}T^2-\|\text{g}_1 q \text{g}_2\|^2)^{\frac{m}{2}}}{t^{m+2}}
dt\right)& \left(1+O\left(\frac{1}{T^{2\e_2-\e_1}}\right)\right). \label{eq:estimate for the fraction in the integral a_de to c}  
\end{align}
We shall now estimate $\int_{\al_\de}^\la \frac{(t^{\frac{2}{m}}T^2-\|\text{g}_1 q \text{g}_2\|^2)^{\frac{m}{2}}}{t^{m+2}}
dt$:
\begin{align}
&\int_{\al_\de}^\la \frac{(t^{\frac{2}{m}}T^2-\|\text{g}_1 q \text{g}_2\|^2)^{\frac{m}{2}}}{t^{m+2}}dt\nonumber\\
=&\int_{\al_\de}^\la\frac{T^m}{t^{m+1}}\left(1-\frac{\|\text{g}_1 q \text{g}_2\|^2}{T^2 t^{\frac{2}{m}}}\right)^{\frac{m}{2}}dt \nonumber\\
=&\frac{m}{2} \frac{T^{m(m+1)}}{\|\text{g}_1 q \text{g}_2\|^{m^2}}\int^{\nicefrac{\|\text{g}_1 q \text{g}_2\|^2}{T^2 \al_{\de}^{\frac{2}{m}}}}_{\nicefrac{\|\text{g}_1 q \text{g}_2\|^2}{T^2 \la^{\frac{2}{m}}}}
u^{\frac{m^2}{2}-1}(1-u)^{\frac{m}{2}}du. \label{upper and lower limits of the integral}
\end{align}It remains to prove the estimate 
\begin{align}
 \int^{\nicefrac{\|\text{g}_1 q \text{g}_2\|^2}{T^2 \al_{\de}^{\frac{2}{m}}}}_{\nicefrac{\|\text{g}_1 q \text{g}_2\|^2}{T^2 \la^{\frac{2}{m}}}}
u^{\frac{m^2}{2}-1}(1-u)^{\frac{m}{2}}du=&\int^{1}_{0}
u^{\frac{m^2}{2}-1}(1-u)^{\frac{m}{2}}du+O\left(\frac{1}{T^{\e_1}}\right)+O\left(\frac{1}{T^{1+\frac{1-2\e_2}{m+1}}}\right)\nonumber\\
=&\frac{\Ga(\frac{m}{2}+1)\Ga(\frac{m^2}{2})}{\Ga(\frac{m^2}{2}+\frac{m}{2}+1)}+O\left(\frac{1}{T^{\e}}\right)
.\label{eq:estimate of the diffence for the beta integral}
\end{align} where in the last equality we used the classical formula for the beta function. Notice that by  \eqref{eq:estimate for the fraction in the integral a_de to c}, by \eqref{upper and lower limits of the integral} and by \eqref{eq:estimate of the diffence for the beta integral} the proof is complete.

To prove \eqref{eq:estimate of the diffence for the beta integral} it suffices to estimate the differences between the endpoints of the corresponding integral, namely it suffices to estimate $\left|{1}-{\frac{\|\text{g}_1 q \text{g}_2\|^2}{T^2 \al_{\de}^{\frac{2}{m}}}}\right|$ and $\left|\frac{\|\text{g}_1 q \text{g}_2\|^2}{T^2\la^{\frac{2}{m}}}-0\right|=\frac{\|\text{g}_1 q \text{g}_2\|^2}{T^2\la^{\frac{2}{m}}}.$ 
We have
\begin{align*}
\left|{1}-{\frac{\|\text{g}_1 q \text{g}_2\|^2}{T^2 \al_{\de}^{\frac{2}{m}}}}\right|=&\left|\frac{T^2\al_\de^\frac{2}{m}-\|\text{g}_1 q \text{g}_2\|^2}{T^2\al_\de^\frac{2}{m}}\right|\\
= &\left|\frac{(T^2\alpha^\frac{2}{m}-\|\text{g}_1 q \text{g}_2\|^2)+(2\al^\frac{1}{m}\de+\de^2)T^2}{T^2\al_\de^\frac{2}{m}}\right|\tag{using definition of $\al_\de$, see \eqref{eq:defintion of a_delta}}\\
= &\frac{\left|O(\nicefrac{1}{T^{m+1}})+(2\al^\frac{1}{m}\de+\de^2)T^2\right|}{T^2\al_\de^\frac{2}{m}}\tag{see  \eqref{eq:asymp for a when A small then sqrt T}}\\
\leq & \frac{\left|O(\nicefrac{1}{T^{m+1}})+(2\al^\frac{1}{m}\de+\de^2)T^2\right|}{\left|T^2\al^\frac{2}{m}\right|}\\
= & O\left(\frac{1}{T^{\e_1}}\right)\tag{since $T^2\al^\frac{2}{m}\asymp \|\text{g}_1q\text{g}_2\|^2$ and $(2\al^\frac{1}{m}\de+\de^2)T^2=O(\frac{\|\text{g}_1q\text{g}_2\|^2}{T^{\e_1}})$},   
\end{align*}
\iffalse
\begin{align}
\left|{1}-{\frac{\|\text{g}_1 q \text{g}_2\|^2}{T^2 \al_{\de}^{\frac{2}{m}}}}\right|=&\left|{1}-\frac{\|\text{g}_1 q \text{g}_2\|^2}{T^2a^{\frac{2}{m}}(1+a^{-\frac{1}{m}}\de)^2}\right|\\
= &\left|{1}-\frac{\|\text{g}_1 q \text{g}_2\|^2}{T^2\left(\frac{\|\text{g}_1 q \text{g}_2\|^2}{T^2}+O(\frac{1}{T^{m+1.5}})\right)(1+a^{-\frac{1}{m}}\de)^2}\right|\\
= & \left|{1}-\frac{\|\text{g}_1 q \text{g}_2\|^2}{T^2\left(\frac{\|\text{g}_1 q \text{g}_2\|^2}{T^2}+O(\frac{1}{T^{m+1.5}})\right)}(1+O(a^{-\frac{1}{m}}\de))\right|\\
= & \left|{1}-\left(1+O(\frac{1}{\|\text{g}_1 q \text{g}_2\|^2T^{m-0.5}})\right)\left(1+O(\frac{T}{\|\text{g}_1 q \text{g}_2\|}\frac{\|\text{g}_1 q \text{g}_2\|}{T^{1+\e_1}})\right)\right|\\
= & O\left(\frac{1}{T^{\e_1}}\right).   
\end{align}
\fi
and finally
\begin{equation*}
  \frac{\|\text{g}_1 q \text{g}_2\|^2}{T^2\la^{\frac{2}{m}}}=\frac{\|\text{g}_1 q \text{g}_2\|^2}{T^2\left(\nicefrac{\|\text{g}_1 q \text{g}_2\|}{T^{\e_2}}\right)^{\frac {2}{m+1}}}=\frac{\|\text{g}_1 q \text{g}_2\|^{2-\frac{2}{m+1}}}{T^{2-\frac{2\e_2}{m+1}}}\underbrace{=}_{\|\text{g}_1 q \text{g}_2\|\leq\sqrt{T}}O\left(\frac{1}{T^{1+\frac{1-2\e_2}{m+1}}}\right).  
\end{equation*}
\end{proof}

We now turn to our estimate for the last integral.
\begin{lemma}\label{lem:integral from c to b}
Let \emph{$\|\text{g}_1 q \text{g}_2\|\leq\sqrt{T}$}, and assume that $g_1,g_2$ vary in a bounded set of $\SL(m+1,\R)$. Fix $\e_2\in(0,1)$. Then \begin{equation*}
\int_{\la(q,T,\e_2)}^{\be_{q,T}} \frac{(R_{q,T}(t))^{\frac{m}{2}}}{t^{m+2}}dt=T^{m+\e_2\frac{ m^2}{m+1}}O\left(\frac{1}{\|q\|^{\frac{m^2}{m+1}}}\right).      
 \end{equation*}
\end{lemma}
\begin{proof}
Again, to ease the reading, we will omit from the notations the dependencies on $q$,$T$ and $\e_2$. 
We have
    \begin{align*}
&\int_\la^\be \frac{(-\det(\text{g}_1\text{g}_2)^{-2} t^{\frac{2}{m}+2}+t^{\frac{2}{m}}T^2-\|\text{g}_1 q \text{g}_2\|^2)^{\frac{m}{2}}}{t^{m+2}}dt\\
\le &\int_\la^\be \frac{(t^{\frac{2}{m}}T^2)^{\frac{m}{2}}}{t^{m+2}} dt =\int_\la^\be \frac{T^m}{t^{m+1}}dt \\
\le & \frac{T^m}{m\la^m}=\frac{T^m}{m(\|\text{g}_1 q \text{g}_2\|/T^{\e_2})^{\frac{m^2}{m+1}}}=\frac{T^{m+\frac{\e_2 m^2}{m+1}}}{m\|\text{g}_1 q \text{g}_2\|^{\frac{m^2}{m+1}}}
\end{align*} 
\end{proof}

\section{Equidistribution along skewed $H$-balls}\label{sec:proof of equid}
Recall $G=\SL(m+1,\R)$, and $\Ga\leq G$ is a lattice. The second important ingredient required for the method of \cite{Gorodnik2004DistributionOL} are the following ergodic theorems.  For  $x\in X,g_1, g_2\in G$  and $T>0$, we consider the measure defined by the integral 
\begin{equation}\label{definition of measure funtional}
    \mu_{T,g_1,g_2}(F):=\frac{1}{\mu \left(H_T[g_1,g_2] \right)}\int_{H_T[g_1,g_2]}F(h^{-1}x)d\mu(h),~F\in C_c(G/\Ga).
\end{equation}
%The following properties follow immediately from \eqref{measure of H_T}:
%\begin{proof}
%It follows from Proposition \ref{computation for H and V} that 
%\begin{equation}\label{definition and computation of alpha}
 %   \omega(g_1^{-1},g_2):=\lim_{T\to \infty}\frac{\mu \left(H_T[g_1,g_2] \right)}{\mu \left(H_T \right)}=\frac{1}{\det(\text{g}_1)^{-m} |\det(\text{g}_2)|}\frac{\sum_{q\in \Delta}\frac{1}{\|\text{g}_1 q \text{g}_2\|^{m^2}}}{\sum_{q\in \Delta}\frac{1}{\|q \|^{m^2}}}.
%\end{equation}
%\end{proof}

\begin{theorem}\label{The $G$-invariance of limiting measure}
Let $\mu_X$ be the  $G$-invariant probability measure on $X=G/\Ga$. For all $x\in X,g_1, g_2\in G$  and all $F\in C_c(X)$
\begin{equation}
    \lim_{T\to \infty} \mu_{T,g_1,g_2}(F)=\mu_X(F).
\end{equation}
\end{theorem}
\iffalse
\textbf{TODO: explain more in detail hwo the ergodic theorem is proved, that it generalizes to other norms, and that it's not possible to use Margulis's banana trick}The above ergodic theorem is proven in Section \ref{sec:proof of equid}. In an overview, we will establish unipotent invariance of the limiting measure, which opens the way to apply the celebrated results of Ratner (see e.g. \cite{Ratner91a}), combined with the results of Shah (see e.g. \cite{Shah1994LimitDO}) on the behaviour of polynomial orbits. The main technical result we prove is a certain divergence of polynomial maps in representation space, where the domain of the maps "shrinks" (see Lemma \ref{our lemma expansion inequality for applying shah dichotomy}).
\fi
We start by reducing Theorem \ref{The $G$-invariance of limiting measure} to the  equidistribution statement along skewed balls of the connected component of $H$. Similarly to the above, consider for $x\in X,g_1,g_2\in G$ and $T>0$ the measures
\begin{equation}
   \label{eq:measures along the connected component} \mu^\circ_{T,g_1,g_2}(F):=\frac{1}{\mu(V_T[g_1,g_2])}\int_{V_T[g_1,g_2]}F(a_t^{-1}u_{v}^{-1}x)dv\frac{1}{t^{m+2}}dt,
\end{equation}
where $V_T[g_1,g_2]$ is the skew-ball of the identity component of $H$.
\begin{theorem}\label{thm:equidisitribution along skewed balls of the connected component}
   For all $x\in X, g_1,g_2\in G$ and all $F\in C_c(X)$ it holds that $$\lim_{T\to\infty}\mu_{T,g_1,g_2}^\circ(F)=\mu_X(F).$$
\end{theorem}
\begin{proof}[Proof of Theorem \ref{The $G$-invariance of limiting measure} assuming Theorem \ref{thm:equidisitribution along skewed balls of the connected component}]
In view of \eqref{decomposition of H_T into V_T}, we can decompose $\mu_{T,g_1,g_2}(F)$ as the following convex linear combination:\begin{align*}
    \mu_{T,g_1,g_2}(F)
    =\sum_{\substack{q\in \Delta}}\frac{\mu \left(V_{q,T}[g_1,g_2] \right)}{\mu \left(H_T[g_1,g_2] \right)}\frac{1}{\mu \left(V_{q,T}[g_1,g_2] \right)}\int_{V_{q,T}[g_1,g_2]} F(\tilde{q}^{-1}a_t^{-1}u_v^{-1}x) \frac{1}{t^{m+2}}dv dt
\end{align*} 
Recall that \begin{equation*}
    \mu(V_{q,T}[g_1,g_2])=\mu(V_T[g_1,\tilde{q} g_2]),
\end{equation*}and observe that \begin{equation}
    \frac{1}{\mu \left(V_{q,T}[g_1,g_2] \right)}\int_{V_{q,T}[g_1,g_2]} F(\tilde{q}^{-1}a_t^{-1}u_v^{-1}x) \frac{1}{t^{m+2}}dv dt=\mu_{T,g_1,\tilde{q} g_2}^\circ(L_{\tilde{q}^{-1}}(F)).
\end{equation}By assuming Theorem \ref{thm:equidisitribution along skewed balls of the connected component}, we have for all $q\in\Delta$,$$\lim_{T\to\infty}\mu_{T,g_1,\tilde{q} g_2}^\circ(L_{\tilde{q}^{-1}}(F))=\mu_X(L_{\tilde{q}^{-1}}(F))\underbrace{=}_{\text{invariance}}\mu_X(F).$$We denote \begin{equation*}
    c_{q,T}:=\frac{\mu \left(V_{q,T}[g_1,g_2] \right)}{\mu \left(H_T[g_1,g_2] \right)}.
\end{equation*}
Then clearly, $$\sum_{q\in\Delta}c_{q,T}=1,~\forall T,\text{ and }c_{q,T}\leq 1,~\forall q,T.$$
Importantly, by Lemma \ref{lem:estimate of V skewed balls as power of T and fract of A_m to power m^2} and by \eqref{eq:the estimate for H skew balls} there is $C>0$ such that for all $T>0$ $$c_{q,T}\leq\frac{C}{\|q\|^{m^2}}.$$
By Lemma \ref{lemma on the integral and summation in a ball}  the sum $\sum_{q\in\Delta}\frac{1}{\|q\|^{m^2}}$ converges. Then for arbitrary small $\e>0$, we may find $N_\e$ such that $$\sum_{\substack{q\in\Delta,\\\|q\|>N_\e}}c_{q,T}\leq \e,~ \forall T>0.$$
Thus, \begin{align*}
    |\mu_{T,g_1,g_2}(F)-\mu_X(F)|=&\left|\sum_{q\in\Delta}c_{q,T}(\mu_{T,g_1,\tilde{q} g_2}^\circ(F)-\mu_X(F))\right|\\
    \leq &\sum_{q\in\Delta}c_{q,T}\left|\mu_{T,g_1,\tilde{q} g_2}^\circ(F)-\mu_X(F)\right|\\
    \leq&\sum_{\substack{q\in\Delta,\\\|q\|\leq N_\e}}\left|\mu_{T,g_1,\tilde{q} g_2}^\circ(F)-\mu_X(F)\right|+2\e\|F\|_\infty.
\end{align*}
Then, by assuming Theorem \ref{thm:equidisitribution along skewed balls of the connected component}, we get that $$\limsup_{T\to\infty}|\mu_{T,g_1,g_2}(F)-\mu_X(F)|\leq2\e\|F\|_\infty,$$which concludes our proof.
\end{proof}

In the rest of the section we will be proving Theorem \ref{thm:equidisitribution along skewed balls of the connected component}, and we now explain the main ideas. In an overview, we prove unipotent invariance of the limiting measure and apply the linearisation technique, which is by now a standard technique in homogeneous dynamics. More specifically, let $\overline{X} = \overline{G/\Ga}$ denote the one-point compactification of $X=G/\Ga$. Then, by the Banach-Alaoglu theorem, there is a weak-* limit $\eta$ of the measures $\mu^\circ_{T,g_1,g_2},~ T>0$, and $\eta$ is a probability measure on $\overline{X}$. In Section \ref{sec:unipotent invariance}, we prove that $\eta$ is invariant under $U=\begin{bmatrix}
    I_m&0\\\R^m&1
\end{bmatrix}$. The celebrated  Ratner theorems on measure rigidity \cite{Ratner91a} state that any $U$-invariant and \emph{ergodic} measure on $G/\Ga$ is supported on a closed orbit having finite measure. In particular, this allows to describe the ergodic decomposition of the $U$-invariant measure $\eta$, and reduces the problem to proving that there is no escape of mass, namely $\eta(\infty)=0$, and to showing that the $\eta$-measure of proper closed orbits is zero. To prove the later fact, notice that the measures $\mu^\circ_{T,g_1,g_2}$ involve integration along families of polynomial trajectories. This allows to utilize deep results due to Shah on the behavior of polynomial trajectories in $G/\Ga$, see \cite{Shah1994LimitDO} and \cite{Shah1996LimitDO} (generalizing  the celebrated results of Dani-Margulis \cite{Dani1993LimitDO}) building on the linearisation technique. The essential challenge in applying the theorems of Shah is proving a certain expansion property of the acting group in representation space. We note that our apporach is similar to the one taken by Gorodnik in  \cite{Gorodnik2003LatticeAO}. The essential difference between our approach compared to \cite{Gorodnik2003LatticeAO} is in our treatment of the expansion in the representation space. 
\begin{remark}
Notice that our measures are explicitly given by $$\mu^\circ_{T,g_1,g_2}(F)= \frac{1}{\mu \left(V_{T}[g_1,g_2] \right)} \int_{\al_T}^{\be_T} \int_{D_{T,t}} F \left(a_t^{-1} u_v^{-1} g_1\Gamma \right) dv \frac{1}{t^{m+2}}dt,~F\in C_c(G/\Ga).$$Now $U$ is the horosphere for $a_t^{-1}$ which expands as $t\to 0.$ Then one might hope to apply the classical Margulis thickening trick (see Margulis's thesis
\cite{mar_thesis}) in order to prove the equidistribution of $\int_{D_{T,t}} F \left(a_t^{-1} u_v^{-1} g_1\Gamma \right) dv$ as $t\to0$. This approach fails since the ellipsoids $D_{T,t}\subseteq\R^{m+1}$ shrink as $t\to0$. Moreover, the integration above is essentially supported on the range where $t$ is small, see \eqref{eq:truncated integral}.
\end{remark}

\subsection{The $U$-invariance of the measure along skewed balls of $V$}\label{sec:unipotent invariance}
The main goal here is to prove that any weak-* limit of $\mu_{T,g_1,g_2}^\circ$  is $U:=\begin{bmatrix}
I_m & 0 \\
\mathbb{\R}^m & 1 
\end{bmatrix}$-invariant. For any $u_0\in U$ and $F\in C_c(X)$ consider $L_{u_0}(F)(x):=F(u_0^{-1}x)$.
\begin{proposition}
\label{unipotent invariance of haar measure}
For all $g_1,g_2\in G$ and all $F\in C_c(X)$ it holds that 
\begin{equation*}
    \lim_{T\to \infty}\left(\mu^\circ_{T,g_1,g_2}(F)-\mu^\circ_{T,g_1,g_2}(L_{u_0} (F)\right)=0.
\end{equation*}
\end{proposition}

The following lemma will be needed: 

\begin{lemma}\label{lemma on symmetric difference}
Let $f:\R^{d}\to \R$ be a bounded continuous function. Let $E\subset \R^{d}$ be an ellipsoid with surface area $S$, then for $y\in \R^{d}$ we have the estimate:
\begin{equation*}
    \left|\int_{E}[f(v)-f(y+v)]dv\right| \ll \|f\|_\infty \|y\|S,
\end{equation*}
where the implied constant depends on $d$ only.
\end{lemma}
\begin{proof}
An elementary geometric argument (thickening the boundary) implies that $$\vol(E\triangle (E-y))\ll \|y\|S,$$where $S$ is the surface area of the ellipsoid, and Vol stands for the Lebesgue volume. Alternatively, see \cite{symmetricdifferencesch2010} for a much more general theorem implying this fact. Now
\begin{align*}
    \left|\int_{E}[f(v)-f(y+v)]dv\right|
  = &\left|\int_{E}f(v)dv - \int_{E}f(v)dv\right|\\
  \le & \int_{E\triangle
 (E-y)}|f(v)|dv\\
  \ll & \|f\|_\infty\|y\|S, \tag{$f$ is bounded}
\end{align*}
\end{proof}
\begin{proof}[Proof of Proposition \ref{unipotent invariance of haar measure}]
Recall that $$V_T[g_1,g_2]=\{u_va_t:t>0,~v\in D_{T,t}\},$$where
\begin{equation}\label{eq:ellipsoid for connected component}
 D_{T,t}:=\{v\in\R^m:\|\text{v}_1 \text{g}_2+\det(\text{g}_1)^{-1} v \text{g}_2+\det(\text{g}_1)^{-1} \text{v}_2t^{\frac{1}{m}+1}\|^2 \leq R_{T}(t)\}, 
\end{equation}and where\begin{equation}\label{eq:radius for connected component}
   R_T(t):=-\det(\text{g}_1\text{g}_2)^{-2} t^{\frac{2}{m}+2}+t^{\frac{2}{m}}T^2-\|\text{g}_1 \text{g}_2\|^2, 
\end{equation}see \eqref{eq:def of R_q,T}, \eqref{eq:def of D} (with $q=I_m$). Here $\text{g}_i,\text{v}_j$ are defined by \eqref{matrix representation of g1 and g2}.  We denote $\al_T:=\al_{I_m,T}$ and $\be_T:=\be_{I_m,T}$ the roots $R_T(t)$, see Lemma \ref{lem:roots lemma} (again, with $q=I_m$). 
Let $F\in C_c(X)$ and fix $u_{v_0}\in U$. We have
\begin{align*}
    &\left| \int_{V_{T}[g_1,g_2]} F(h^{-1}x) d\mu(h)-\int_{V_{T}[g_1,g_2]} F(u_{v_0}^{-1}h^{-1}x) d\mu(h) \right |\\
    =&\left|\int_{\al_{T}}^{\be_{T}}\int_{D_{T,t}} \left(F( a_t^{-1} u_v^{-1} x) - F(u_{v_0}^{-1} a_t^{-1} u_v^{-1}x) \right) dv \frac{1}{t^{m+2}}dt\right |
    \end{align*}
\begin{align*}=&\left|\int_{\al_{T}}^{\be_{T}} \int_{D_{T,t}} \left(F( a_t^{-1} u_v^{-1} x) - F(a_t^{-1} [a_t u_{v_0}^{-1}a_t^{-1}]  u_v^{-1}x) \right) dv \frac{1}{t^{m+2}}dt\right |\\
    =&\left|\int_{\al_{T}}^{\be_{T}} \int_{D_{T,t}} \left(F( a_t^{-1} u_{-v} x) - F(a_t^{-1} u_{-t^{\frac{1}{m}+1}v_0-v}  x) \right) dv \frac{1}{t^{m+2}}dt\right |.
\end{align*}

For the use of Lemma \ref{lemma on symmetric difference}, notice that the surface area of the ellipsoids $D_{T,t}$ is $\asymp (R_T(t))^{\frac{m-1}{2}}$. We have

\begin{align*}
    &\left| \int_{\al_T}^{\be_T}  \int_{D_{T,t}} \left(F( a_t^{-1} u_{-v} g_1\Gamma) - F(a_t^{-1} u_{-t^{\frac{1}{m}+1}v_0-v}  g_1\Gamma) \right) dv \frac{1}{t^{m+2}}dt \right|\\
    \ll &  \int_{\al_T}^{\be_T}   \|-t^{\frac{1}{m}+1}v_0\|(-\det(\text{g}_1\text{g}_2)^{-2} t^{\frac{2}{m}+2}+t^{\frac{2}{m}}T^2-\|\text{g}_1  \text{g}_2\|^2)^\frac{m-1}{2} \frac{1}{t^{m+2}}dt  \tag{by Lemma \ref{lemma on symmetric difference}}\\
    \ll &  \int_{\al_T}^{\be_T}   \frac{(t^{\frac{2}{m}}T^2)^\frac{m-1}{2}}{t^{m-\frac{1}{m}+1}}dt \\
    =&T^{m-1}\int_{\al_T}^{\be_T}\frac{1}{t^m}dt\\
    \ll& T^{m-1}\frac{1}{\al_T^{m-1}}\\
    \ll& T^{m-1}\frac{1}{1/T^{m(m-1)}}=T^{m^2-1}. \tag{by \eqref{range for a}}
\end{align*}
After normalization by $V_{T}[g_1,g_2]\asymp T^{m(m+1)}$, this term goes to zero as $T\to \infty$.
\end{proof}

\subsection{Expanding property of $UA$ in representations of $\SL(m+1,\R$)}
In order to apply the results of Shah, we need to establish that vectors expand by the action of $V=UA$ in a representation of $G=\SL(m+1,\R)$. We denote by $\V_G$ a representation of $G$ and we let $$\V_G=\V_0\oplus \V_1,$$ where $\V_0$ is the space of $\SL(m+1,\R)$-fixed vectors and $\V_1$ is it's invariant complement. Suppose that $\V_1$ is not trivial, and let $\Pi$ be the projection of $\V_G$ onto $\V_1$ with kernel $\V_0$. Our main result in this section is the following.
\begin{lemma}\label{our lemma expansion inequality for applying shah dichotomy} Fix $g_0\in G$, let $K\subseteq \R^m$ be a bounded set, and let $\chi\in(0,1)$. Then for any $\nu>0$, there exist $t_0>0$ such that for any $0<t<t_0$, any $\xi\in K$ and any \emph{$\textbf{x}\in \V_G$} such that \emph{$\|\Pi(\textbf{x})\|\geq r$}, it holds
\emph{\begin{equation}\label{expansion inequality for applying shah dichotomy}
    \sup_{v\in \text{B}\left(t^{\frac{1-\chi}{m}}\right)+\xi} \|a_t^{-1}u_v^{-1}g_0.\textbf{x}\| >\nu,
\end{equation}}
where \emph{$\text{B}(\beta)$} is the ball of radius $\beta$ in $\R^m$ centered at the origin. 
\end{lemma}

The idea behind the proof of the lemma to decompose the elements  $a_t^{-1}u_v^{-1}$ into  $m$ elements  each belonging to a subgroup isomorphic to $\SL(2,\R)$. That's useful since we can use the explicit description of irreducible representations  of $\SL(2,\R)$ to prove an   expansion of the norm for each term in the decomposition. We note that a similar approach was used in \cite{Kleinbock2006DIRICHLETSTO}.

We now give the decomposition. Let  $E_{i,j}$ be the $(m+1)\times(m+1)$ matrix where the $(i,j)$-th entry is $1$ while all other entries are zero. Consider the following copy of $\SL(2,\R)$ in $\SL(m+1,\R)$  
\begin{equation}\label{eq:j-th copy of SL2 in SLm+1}
\SL^{(j)}(2,\R):=\{I_{m+1}+(a-1)E_{jj}+bE_{m+1,j}+cE_{j,m+1}+(d-1)E_{m+1,m+1}:ad-bc=1\},
\end{equation}
for $j=1,2,...,m$. We will denote \begin{equation*}
    \pi^{(j)}\begin{bmatrix}
        a&b\\c&d
    \end{bmatrix}:=I_{m+1}+(a-1)E_{jj}+bE_{m+1,j}+cE_{j,m+1}+(d-1)E_{m+1,m+1}.
\end{equation*}
We have the following observation which we leave the reader to verify.
\begin{lemma}\label{change of order lemma}
    Let $\sigma:\{1,2,\dots m\}\to \{1,2,\dots,m\}$ be a permutation. Then it holds that \begin{equation}\label{eq:decompostion of AU to SL2 mats}
   a_t^{-1}u_v^{-1}=\prod_{j=1}^m\pi^{(\sigma(j))}\left(
   \begin{bmatrix}
   t^{\frac{1}{m}}&0\\
   0&t^{-\frac{1}{m}}
 \end{bmatrix}
 \begin{bmatrix}
   1&0\\
   -t^{-\frac{m-j}{m}}v_{\sigma(j)}&1
 \end{bmatrix}\right),   
    \end{equation}
  for $a_t,u_v$ as in \eqref{eq:definition of u_v a_t and qtilde}.
\end{lemma}
 
The decomposition \eqref{eq:decompostion of AU to SL2 mats} reduces the proof of  Lemma \ref{our lemma expansion inequality for applying shah dichotomy} to the study of the expansion of elements of the form $$\begin{bmatrix}
t^{\frac{1}{m}} & 0 \\
0 & t^{-\frac{1}{m}} 
\end{bmatrix}\begin{bmatrix}
1 & 0 \\
y & 1 
\end{bmatrix}$$ in representations of $\SL(2,\R)$. 
\subsubsection{Expansion in $\SL(2,\R)$-irreducible representations}
We now briefly recall some basic facts on $\SL(2,\R)$-irreducible representations.

Recall that if $\pi$ is an $(n+1)$-dimensional irreducible representation of $\SL(2,\R)$ and $\pi'$ is the induced Lie algebra representation, then there exists a basis $v_0,v_1,...,v_n$ such that
\begin{align*}
    &\pi'(H)(v_i)=(n-2i)v_i, i=0,1,...,n;\\
    &\pi'(X)(v_i)=i(n-i+1)v_{i-1}, i=0,1,...,n;\\
    &\pi'(Y)(v_i)=v_{i+1}, i=0,1,...,n~(v_{n+1}=0).
\end{align*}
where
$H = \begin{bmatrix}
    1 & ~~0\\
    0 & -1
  \end{bmatrix},
  X = \begin{bmatrix}
    0 & 1\\
    0 & 0
  \end{bmatrix}$, and
  $Y = \begin{bmatrix}
    0 & 0\\
    1 & 0
  \end{bmatrix}$ form a generating set of $\frak{sl}(2,\R)$.

\vspace{5mm}
 Under the matrix Lie group-Lie algebra correspondence, 
\[\pi \left(\begin{bmatrix}
    t^{\frac{1}{m}} & 0\\
    0 & t^{-\frac{1}{m}}
\end{bmatrix} \right)=\exp(\pi'(\log(t^{\frac{1}{m}})H))\] has the matrix representation 
\begin{equation}\label{matrix rep of a_t under the basis}
    \begin{bmatrix}
    t^{\frac{n}{m}} & & &\\
    & t^{\frac{n-2}{m}} & & \\
    & & \ddots &\\
    & & & {t}^{-\frac{n}{m}}
  \end{bmatrix}
\end{equation}
and 
\[\pi \left(\begin{bmatrix}
    1 & 0\\
    -y & 1
\end{bmatrix} \right)=\exp(\pi'(-yY))\]
has the matrix representation
\begin{equation}\label{matrix rep of u_y under the basis}
    \begin{bmatrix}
    1 & & &\\
    p_{1}(y) & 1 & & \\
    \vdots   & \ddots & \ddots &\\
    p_{n}(y) & \cdots &  p_{1}(y)  &  1
  \end{bmatrix}
\end{equation}
where $p_{l}(y)=\frac{(-y)^l}{l!},~l=1,...,n$. 
Observe that the line $$W:=\R v_n$$ is the subspace of fixed vectors of $\pi \left(\begin{bmatrix}
    1 & 0\\
    -y & 1
\end{bmatrix}\right),y\in\R.$ This space is important for us since it's the  eigenspace of the matrices $\pi \left(\begin{bmatrix}
    t^{\frac{1}{m}} & 0\\
    0 & t^{-\frac{1}{m}}
\end{bmatrix}\right)$ which corresponds to the largest eigenvalue as $t\to 0$. 
We denote by $\pr_W$ the orthogonal projection on $W$ with respect to some scalar product on $\V_n$. 

The following lemma is a very special case of \cite[Lemma 13]{Gorodnik2003LatticeAO} (which is essentially \cite[Lemma 5.1]{Shah1996LimitDO}). In our setting the proof simplifies considerably and we prove it for completeness.
\begin{lemma}\label{very important inequaity lemma of nimish-gorodnik}
Let $\pi:\SL(2,\R)\to \V_n$ be the $(n+1)$-dimensional irreducible representation of $\SL(2,\R)$.   Fix a bounded interval \emph{I}. Then there exists a constant $C>0$ such that for any $\beta \in (0,1)$, $\tau\in \emph{\text{I}}$ and \emph{$\textbf{x}\in \V_n$},
\emph{\begin{equation*}
    \sup_{y\in([0,\beta]+\tau)}\left\|\pr_W\left(\pi \left(\begin{bmatrix}
    1 & 0\\
    -y & 1
\end{bmatrix}\right)\textbf{x}\right)\right\| \ge C \beta^n\|\textbf{x}\|.
\end{equation*}}
\end{lemma}
\begin{proof}
First, note that it's enough to prove that there exists $C_0>0$ such that:
\begin{equation}\label{eq:expansion of the norm removing tau}
\sup_{y\in[0,\beta]}\left\|\pr_W\left(\pi \left(\begin{bmatrix}
    1 & 0\\
    -y & 1
\end{bmatrix}\right)\textbf{x}\right)\right\|\geq C_0\|\textbf{x}\|.  
\end{equation}
In fact, by assuming \eqref{eq:expansion of the norm removing tau}, we get
\begin{align*}
 \sup_{y\in([0,\beta]+\tau)}\left\|\pr_W\left(\pi \left(\begin{bmatrix}
    1 & 0\\
    -y & 1
\end{bmatrix}\right)\textbf{x}\right)\right\| =&\sup_{y\in[0,\beta]}\left\|\pr_W\left(\pi \left(\begin{bmatrix}
    1 & 0\\
    -y & 1
\end{bmatrix}\right)\pi \left(\begin{bmatrix}
    1 & 0\\
    -\tau & 1
\end{bmatrix}\right)\textbf{x}\right)\right\|\\
&\geq C_0\left\|\pi \left(\begin{bmatrix}
    1 & 0\\
    -\tau & 1
\end{bmatrix}\right)\textbf{x}\right\|\\
&\geq C_0 C_1 \|\textbf{x}\|,
\end{align*}
where $C_1$ depends on I.
Now, in the coordinates of the basis $\{v_0,...,v_n\}$, upon recalling \eqref{matrix rep of u_y under the basis} we see
\begin{align*}
  \pr_W\left(\pi \left(\begin{bmatrix}
    1 & 0\\
    -y & 1
\end{bmatrix}\right)\textbf{x}\right)&=\begin{bmatrix}
    1 & & &\\
    p_{1}(y) & 1 & & \\
    \vdots   & \ddots & \ddots &\\
    p_{n}(y) & \cdots &  p_{1}(y)  &  1
  \end{bmatrix}\begin{bmatrix}
      x_0\\
      x_1\\
      \vdots\\
      x_n
  \end{bmatrix}\\
  &=\begin{bmatrix}
      0\\
      \vdots
      \\
      0\\
      p_{\textbf{x}}(y)
  \end{bmatrix},
\end{align*}
where $p_{\textbf{x}}(y)=c\sum_{k=0}^n\frac{(-1)^k}{k!}x_{n-k}y^k$ (here $c$ is some constant depending on chosen scalar product for the orthogonal projection $\text{pr}_W.$)  Note the general fact that there exists $C_2>0$ such that for any polynomial $p(x)=a_nx^n+...+a_0$ of degree at most $n$  $$\sup_{s\in[0,1]}|p(s)|\geq C_2\|\textbf{a}\|,$$
where $\textbf{a}$ is the vector whose entries are the coefficients of $p(x)$. Thus we may conclude:$$\sup_{y\in[0,\beta]}|p_{\textbf{x}}(y)|=\sup_{s\in[0,1]}\left|\sum_{k=0}^n\frac{(-1)^k}{k!}\beta^k x_{n-k}s^k\right|\gg \beta^n\|\textbf{x}\|,$$
which proves the lemma.
\end{proof}

The above lemma implies the following statement, which is an analogue of Lemma \ref{our lemma expansion inequality for applying shah dichotomy} for $\SL(2,\R)$. 
\begin{lemma}\label{lem:main statement for expansion of sl2 reps}
 Suppose that $\pi:\SL(2,\R)\to\V$ is a representation, fix an interval \emph{I} and let $\chi\in (0,1)$. Then, for all $t\in(0,1)$ the following holds.
 \begin{enumerate}
     \item\label{enu:when pi is irreduc} If $\pi$ is irreducible, then for all $\tau\in\emph{\text{I}}$ it holds that\emph{\begin{equation}\label{eq:expansion of sl(2,r) reps of y in t interval}
    \sup_{y\in[0,t^{\frac{1-\chi}{m}}]+\tau}\left \|\pi\left(\begin{bmatrix}
t^{\frac{1}{m}} & 0 \\
0 & t^{-\frac{1}{m}} 
\end{bmatrix}
\begin{bmatrix}
1 & 0 \\
-y & 1 
\end{bmatrix}\right)\textbf{x} \right \|\gg t^{-\frac{\chi}{m}}\|\textbf{x}\|,
\end{equation}}for all $\emph{\textbf{x}}\in\V$, where the implied constant depends on $\pi$ only.
\item\label{enu:when pi not necessearily irred} If $\pi$ is any representation, then for all $\tau\in \emph{\text{I}}$ \emph{\begin{equation}\label{eq:expansion of sl(2,r) in not necessarily irreducible}
    \sup_{y\in[0,t^{\frac{1-\chi}{m}}]+\tau}\left \|\pi\left(\begin{bmatrix}
t^{\frac{1}{m}} & 0 \\
0 & t^{-\frac{1}{m}} 
\end{bmatrix}
\begin{bmatrix}
1 & 0 \\
-y & 1 
\end{bmatrix}\right)\textbf{x} \right \|\gg \|\textbf{x}\|,
\end{equation}}
for all $\emph{\textbf{x}}\in\V$, where the implied constant depends on $\pi$ only.
 \end{enumerate} 
\end{lemma}
\begin{proof}
    In the following, for convenience, we will omit the representation symbol $\pi$. Suppose that $\V$ is the $n+1$'th irreducible representation. Then for all $y\in \R$ and all $\textbf{x}\in \V$
\begin{align}
& \left \|\begin{bmatrix}
t^{\frac{1}{m}} & 0 \\
0 & t^{-\frac{1}{m}} 
\end{bmatrix}
\begin{bmatrix}
1 & 0 \\
-y & 1 
\end{bmatrix}.\textbf{x} \right \| \nonumber \\
\ge & C \left \| \pr_W\left(\begin{bmatrix}
t^{\frac{1}{m}} & 0 \\
0 & t^{-\frac{1}{m}} 
\end{bmatrix}
\begin{bmatrix}
1 & 0 \\
-y & 1 
\end{bmatrix}.\textbf{x} \right) \right \| \tag{by the boundedness of the linear operator $\pr_W$}\\
= & C \left \| \begin{bmatrix}
t^{\frac{1}{m}} & 0 \\
0 & t^{-\frac{1}{m}} 
\end{bmatrix}.
 \pr_W\left(
\begin{bmatrix}
1 & 0 \\
-y & 1 
\end{bmatrix}.\textbf{x} \right) \right \| \tag{since the matrix \eqref{matrix rep of a_t under the basis} commutes with $\pr_W=\pr_{\R e_n}$}\\
=& C t^{-\frac{n}{m}} \left \|
 \pr_W\left(\begin{bmatrix}
     1&0\\-y&0
 \end{bmatrix}\textbf{x} \right) \right \| \tag{see the last component of \eqref{matrix rep of a_t under the basis}}
\end{align}
Then, part 1 of our corollary follows by Lemma \ref{very important inequaity lemma of nimish-gorodnik}. To conclude part 2, note that we may always decompose $\V$ as $\V_0\oplus\V_1$ where $\V_0$ is the space of fixed vectors, and $\V_1$ is an invariant complement. By further decomposing $\V_1$ into irreducible representations, and applying part 1 of the corollary in each component, we get the result.
\end{proof}
\subsubsection{Proving Lemma \ref{our lemma expansion inequality for applying shah dichotomy}}
Recall $\V_G=\V_0\oplus \V_1$ where $\V_0$ is the space of $\SL(m+1,\R)$-fixed vectors, and $\V_1$ is it's invariant complement. Note that in order to prove Lemma \ref{our lemma expansion inequality for applying shah dichotomy}, it's enough to verify that for a fixed $r>0$, a bounded set $K\subset \R^m$ and $\chi\in(0,1)$ it holds that \begin{equation}\label{eq:what we need to show for lem 3.6}
    \min_{\textbf{x}\in \V_1,~\|\textbf{x}\|\geq r}\sup_{v\in \text{B}\left(t^{\frac{1-\chi}{m}}\right)+\xi} \|a_t^{-1}u_v^{-1}\textbf{x}\|\gg t^{-\chi/m},~\forall\xi\in K
\end{equation}
where $t\in(0,1)$. As $\V_1$ is invariant by the induced action of $\SL^{(j)}(2,\R)$ (see \eqref{eq:j-th copy of SL2 in SLm+1}) for $j=1,2,...,m$, we get the decompositions
\begin{align}
    \V_1 &= \V_{0}^{(j)} \oplus \V_{1}^{(j)},j=1,2,...,m, 
\label{decompositive of vs w.r.t. SL2}
\end{align}where $\V_0^{(j)}$ is the subspace of fixed vectors of the action of $\SL^{(j)}(2,\R)$ and $\V_1^{(j)}$ is its invariant complement. 
For any $\textbf{x}\in \V_1$ we write 
\begin{equation}    \textbf{x}=\textbf{x}_{0}^{(j)}+\textbf{x}_{1}^{(j)}, j=1,2,...,m.
\end{equation}
where $\textbf{x}_{i}^{(j)} \in \V_{i}^{(j)}, i=1,2,...,m$ from above. 

\begin{lemma} It holds that
\emph{$\inf_{\textbf{x}\in \V_1, \|\textbf{x}_1\|\ge r} \left( \|\textbf{x}_{1}^{(1)}\|+\cdots+\|\textbf{x}_{1}^{(m)}\| \right) = Lr$} for some $L>0$
\end{lemma}
\begin{proof}
We recall that $UA$ (see \eqref{eq:def of the subgroups UADelta}) is an  epimorphic subgroup, cf. \cite{Shah2000OnAO}. This simply means that if $\textbf{x}\in \V_1$ is fixed by $UA$, then it's fixed by all of $\SL(m+1,\R)$. But in such case, $\textbf{x}=0$, since $\V_1$ is an invariant complement of the fixed vectors. In particular, this implies that $\bigcap_{i=1}^m\V_{0}^{(i)}=\{0\}$ using the decomposition in Lemma \ref{change of order lemma}. Since the sphere of radius $1$ is compact and as $\bigcap_{i=1}^m\V_{0}^{(i)}=\{0\}$, we have $$\inf_{\textbf{x}_1\in \V_1, \|\textbf{x}_1\|= 1} \left( \|\textbf{x}_{1}^{(1)}\|+\cdots+\|\textbf{x}_{1}^{(m)}\| \right) = L > 0.$$ 
By scaling, it is easy to see for $r \ge 1$, 
\begin{align*}
    \inf_{\textbf{x}_1\in \V_1, \|\textbf{x}_1\|= r} \left( \|\textbf{x}_{1}^{(1)}\|+\cdots+\|\textbf{x}_{1}^{(m)}\| \right) 
%    =& \frac{r}{1}\inf_{w_1\in V_1, \|w_1\|= r} \left( \|w_{1}^{(1)}\|+\|w_{1}^{(2)}\| \right) \\
    = & r \inf_{\textbf{x}_1\in \V_1, \|\textbf{x}_1\|= 1} \left( \|\textbf{x}_{1}^{(1)}\|+\cdots+\|\textbf{x}_{1}^{(m)}\| \right)  \\
    =& Lr>0.
\end{align*}
\end{proof}
Now take $\textbf{x}\in \V_1$ and fix $j$ such that $\|\textbf{x}^{(j)}_1\|\geq Lr$. We pick a permutation $\sigma:\{1,...,m\}\to\{1,...,m\}$ such that $\sigma(m)=j$. Using Lemma \ref{change of order lemma} we write \begin{align}\label{eq:chose permutation for j with large component}
a_t^{-1}u_v^{-1}\textbf{x}=\prod_{l=1}^{m-1}&\pi^{(\sigma(l))}\left(
   \begin{bmatrix}
   t^{\frac{1}{m}}&0\\
   0&t^{-\frac{1}{m}}
 \end{bmatrix}
 \begin{bmatrix}
   1&0\\
   -t^{-\frac{m-l}{m}}v_{\sigma(l)}&1
 \end{bmatrix}\right)
 \pi^{(j)}
 \left(
   \begin{bmatrix}
   t^{\frac{1}{m}}&0\\
   0&t^{-\frac{1}{m}}
 \end{bmatrix}
 \begin{bmatrix}
   1&0\\
   -v_{j}&1
 \end{bmatrix}\right)\textbf{x}.
\end{align}
For a bounded set $K\subseteq \R^m$ there exist intervals $\text{I}_1,...,\text{I}_m$ such that $$\left(\left[0,\frac{1}{\sqrt{m}}t^{\frac{1-\chi}{m}}\right]+\text{I}_1 \right)\times...\times \left(\left[0,\frac{1}{\sqrt{m}}t^{\frac{1-\chi}{m}}\right]+\text{I}_m \right)\subseteq \text{B}\left(t^{\frac{1-\chi}{m}}\right)+K.$$Then \eqref{eq:what we need to show for lem 3.6} is obtained by applying Lemma \ref{lem:main statement for expansion of sl2 reps} as follows: We decompose $\V_1^{(j)}$ into $\SL^{(j)}(2,\R)$-irreducible representations, and we assume (without loss of generality, as all norms are equivalent) that our norm $\|\cdot\|$ on $\V_1$ is obtained by taking the sup norm with respect to a basis of $\V_1$ composed out of a basis for $\V_0^{(j)}$ and bases for each of the $\SL^{(j)}(2,\R)$-irreducible spaces.  Then, using Lemma \ref{lem:main statement for expansion of sl2 reps},\eqref{enu:when pi is irreduc}, we have for all $\tau_j\in \text{I}_j$  that
\begin{equation*}
     \sup_{v_j\in[0,t^{\frac{1-\chi}{m}}]+\tau_j}\left \|\pi^{(j)}\left(\begin{bmatrix}
t^{\frac{1}{m}} & 0 \\
0 & t^{-\frac{1}{m}} 
\end{bmatrix}
\begin{bmatrix}
1 & 0 \\
-y & 1 
\end{bmatrix}\right)\textbf{x} \right \|\gg t^{-\frac{\chi}{m}}r,
\end{equation*}
and by further applying Lemma \ref{lem:main statement for expansion of sl2 reps},\eqref{enu:when pi not necessearily irred} when taking the supremum of \eqref{eq:chose permutation for j with large component} over the parameters $\left(\left[0,\frac{1}{\sqrt{m}}t^{\frac{1-\chi}{m}}\right]+\text{I}_1 \right)\times...\times \left(\left[0,\frac{1}{\sqrt{m}}t^{\frac{1-\chi}{m}}\right]+\text{I}_m \right)$, we obtain \eqref{eq:what we need to show for lem 3.6}. 

 \subsection{The non-escape of mass}\label{sec:non escape of mass}
Let $\frak g$ be the Lie algebra of $G=\SL(m+1,\R)$. For positive integers $d$ and $n$, denote by
$\mathcal {P}_{d,n}(G)$ the set of functions $\varphi: \R^n \to G$ such that for any $\textbf{a}, \textbf{b} \in \R^n$, the map
\begin{equation*}
   \R\ni \tau \mapsto \Ad(\varphi(\tau \textbf{a} +\textbf{b}))\in \GL(\frak{g})
\end{equation*}
is a polynomial of degree at most $d$ when the above linear transformation is represented by some basis of $\frak{g}$.

Let $\V_G = \sum_{i=1}^{\dim \frak g} \bigwedge^i \frak{g}$. We consider the natural action of $G$ on $\V_G$ induced from the Ad-representation.  Fix a norm $\|\cdot\|$ on $\V_G$. For a Lie subgroup H of $G$ with Lie algebra $\frak{h}$, take a unit vector $p_\text{H} \in \bigwedge^{\dim \frak h} \frak{h}$.

\begin{theorem}[Special case of the theorems 2.1 and 2.2 in \cite{Shah1996LimitDO}, combined]\label{Shah dichotomy theorem}
There exist closed subgroups $U_i (i=1,2,...,l)$ such that each $U_i$ is the unipotent radical of a parabolic subgroup, $U_i\Ga$ is compact in $X=G/\Ga$ and for any $d,n\in \N$, $\e,\de>0$, there exists a compact set $C \subset G/\Ga$ such that for any $\varphi \in \mathcal{P}_{d,n}(G)$ and a bounded open convex set $D\subset \R^n$, one of the following holds:
\begin{enumerate}
    \item there exist $\gamma\in \Ga$ and $i=1,...,l$ such that $\sup_{v\in D}\|\varphi(v)\gamma.p_{U_i}\| \le \de$; 
    \item $\vol(v\in D:\varphi(v)\Ga \notin C)< \e \vol(D)$, where $\vol$ is the Lebesgue measure on $\R^n$.
\end{enumerate}
\end{theorem}
%For $q \in \Ga=\SL(3,\Z)$, we shall consider the following polynomial map: Let $n=2, t>0$, $v\in \R^m$ and define $q(t,v)=q(t,v,q):=q a_t^{-1} u_v^{-1}$. 

Fix  $g_1,g_2\in\SL(m+1,\R)$. We shall consider the family of polynomial maps  appearing in the integration along the measures $\mu^\circ_{T,g_1,g_2}$ which are given by \begin{equation}\label{eq:definition of q t v}
   \varphi_t(v):= a_t^{-1} u_v^{-1} g_0,
\end{equation}
where $g_0\in \SL(m+1,\R)$,  $t>0$ and $v\in\R^m$.
For each fixed $t$, the map $\varphi_t(\cdot)$ is in $\mathcal{P}_{2,m}(G)$. 
Using Lemma \ref{our lemma expansion inequality for applying shah dichotomy}, we will show that $\varphi_t(v)$ fails the condition 1 of the Theorem \ref{Shah dichotomy theorem} in the range appearing in the integration along $\mu^\circ_{T,g_1,g_2}$. This will be the key fact that allows to prove the non-escape of mass and also the equidistribution, see Section \ref{proof of equidistribution}. 

We have the following statement which will enable to verify the requirement of Lemma \ref{our lemma expansion inequality for applying shah dichotomy} for a lower bound of the norm of the projection to the space of non-fixed vectors of the vectors $\gamma.p_{U_i},\gamma\in\Ga,i=1,...,l.$
%Extend representation $\pi:G \to \text{GL}(\V_G)$ as in the theorem above to a complex representation by replacing $\V_G$ with $\V_G \otimes \C$, and 

Denote by $\mathcal{H}_{\Ga}$ the family of all proper closed connected subgroups H of $G$ such that $\Ga \cap \text{H}$ is a lattice in H, and $\Ad(\text{H} \cap \Ga)$ is Zariski-dense in $\Ad(\text{H})$. We have the following important theorem:

\begin{theorem}[Proposition 2.1  and Theorem 3.4 in \cite{Dani1993LimitDO}]\label{ratner's theorem on discreteness}
The set $\mathcal{H}_{\Ga}$ is countable. For any $\emph{\text{H}}\in \mathcal{H}_{\Ga}, \Ga.p_\emph{\text{H}}$ is discrete.    
\end{theorem}

Consider the finite set of unipotent radicals of parabolic subgroups, denoted $U_1,...,U_l$ appearing in Theorem \ref{Shah dichotomy theorem}.  By Theorem \ref{ratner's theorem on discreteness}, we have that  $\Ga.p_{U_i}$ is discrete in $\V_G$ for all $i$.

Write $\V_G=\V_0 \bigoplus \V_1$, where $\V_0$ is the space of vectors fixed by $G$ and $\V_1$ is its $G$-invariant complement. Recall that $\Pi$ denotes the projection of $\V_G$ onto $\V_1$ with kernel $\V_0$. By Lemma 17 in \cite{Gorodnik2003LatticeAO} it holds that $\Pi(\Ga.p_{U_i})$ is discrete for all $i$. Since $p_{U_i}$ is not fixed by $G$ (this is because the action is through conjugation and $U_i$'s are not normal subgroups in the simple group $G$), $\Pi(\Ga.p_{U_i})$ does not contain $0$. So it follows that 
\begin{equation}\label{inf of the projection is at least r}
    \inf_{\textbf{x}\in \cup_{i=1}^l \Ga.p_{U_i}} \|\Pi(\textbf{x})\|:=r>0.
\end{equation}

\subsubsection{Proof of $\eta{(\infty)}=0$}\label{section 4.2}

\iffalse
Therefore, for any $x\in \cup_{i=1}^r \Ga.p_{U_i}$,
\begin{equation}
    \lim_{t\to 0} \sup_{v\in D(\frac{1}{2}\st^{\frac{1-\chi}{m}})} \|a_t^{-1}u_v^{-1}g_1.x\| =\infty.
\end{equation}
\fi

Recall that $\eta$ is a weak-* limit of the measures
\begin{align*}
    \mu^\circ_{T,g_1,g_2}(F):=& \frac{1}{\mu \left(V_{T}[g_1,g_2] \right)} \int_{\al_T}^{\be_T} \int_{D_{T,t}} F \left(a_t^{-1} u_v^{-1} g_0\Gamma \right) dv \frac{1}{t^{m+2}}dt,~F\in C_c(G/\Ga),
\end{align*}
where $0<\al_T<\be_T$ are the two positive roots of \eqref{eq:radius for connected component} (see Lemma \ref{lem:roots lemma} with $q=I_m$), and $D_{T,t}$ is the ellipsoid \eqref{eq:ellipsoid for connected component}. We note that the ellipsoid $D_{T,t}$ is centered at 
\begin{equation}\label{eq:def of translation}
  \xi_T(t):=\det(\text{g}_1)\text{v}_1 -\text{v}_2\text{g}_2^{-1}t^{\frac{1}{m}+1} \in \R^m.  
\end{equation}Then  $D_{T,t}$ contains the ball $\text{B}(\rho_T(t))+\xi_T(t)$ with the same center whose radius $\rho_T(t)$ is equal to the \textit{shortest radius} of the ellipse, and $\rho_T$ has the estimate:
\begin{equation}\label{eq:radius of the contained ball}
\rho_T(t):=C_{g_1,g_2}\sqrt{R_T(t)}=C_{g_1,g_2}\sqrt{ -\det(\text{g}_1\text{g}_2)^{-2} t^{\frac{2}{m}+2}+t^{\frac{2}{m}}T^2-\|\text{g}_1 \text{g}_2\|^2},
\end{equation}
for some constant $C_{g_1,g_2}>0$. The  displacement  $\xi_T(t)$ has two parts: one part ($-\det(\text{g}_1)\text{v}_1$) is constant and the other part ($-\text{v}_2\text{g}_2^{-1}t^{\frac{1}{m}+1}$) involves $t$ but is  bounded when $t\in(0,1)$. To use Lemma \ref{our lemma expansion inequality for applying shah dichotomy} in our proof for $\eta(\infty)=0$, we require that $\rho_{T}(t)>t^{\frac{1-\chi}{m}}$. But, when $t$ approaches to either $\al_T$ or to $\be_T$, $\rho_{T}(t)$ becomes too small. Thus we use the following truncation -- for a fixed  $F\in C_c(G/\Ga)$, $\e_1,\e_2\in(0,1)$ there is $\ka=\ka(\e_1,\e_2)>0$ such that \begin{equation}\label{eq:truncated integral}
\mu^\circ_{T,g_1,g_2}(F)= \frac{1}{\mu \left(V_{T}[g_1,g_2] \right)} \int_{\al_{\de}(T,\e_1)}^{\la(T,\e_2)} \int_{D_{T,t}} F \left(a_t^{-1} u_v^{-1} g_0\Gamma \right) dv \frac{1}{t^{m+2}}dt+O(T^{-\ka}).
\end{equation}
The above estimate follows by Lemmata \ref{lemma for integral from a to a delta} and \ref{lem:integral from c to b} together with the estimate of $\mu(V_T[g_1,g_2])\asymp T^{m(m+1)}$.
Here $\al_{\de}\text{ and }\la$ are given by \eqref{eq:defintion of a_delta} and \eqref{eq:definition of c} with $q=I_m.$
\begin{lemma}\label{lem:the domain Dt,T has a ball of sufficiently large radius in it}
  Fix $\chi\in(0,1)$. Then, there exist $0<\e_1<\e_2<1$ such that for all $T$ large enough and all $t\in(\al_{\de}(T,\e_1),\la(T,\e_2))$, it holds that
   \begin{equation*}
   \rho_{T}(t)>t^{\frac{1-\chi}{m}}.  
   \end{equation*} 
  \iffalse
  For all $T$ large enough its holds that for all $t\in(\al_{\de,T},\la_T)$ 
   \begin{equation}
     \sqrt{ -\det(\text{g}_1\text{g}_2)^{-2} t^{\frac{2}{m}+2}+t^{\frac{2}{m}}T^2-\|\text{g}_1 \text{g}_2\|^2}\geq t^{\frac{1-\chi}{m}}  
   \end{equation}
   \fi
\end{lemma}
\begin{proof}
To prove the inequality, we bound from below the left hand side, and bound from above the right hand side. For the right hand side, by recalling  \eqref{eq:definition of c} defining  $\la$, we get for $t\in(\al_{\de}(T,\e_1),\la(T,\e_2))$ that
\begin{equation}\label{eq:asymp for righthand side}
    \frac{1}{T^\frac{\e_2(1-\chi)}{m+1}}\asymp \la(T,\e_2)^{\frac{1-\chi}{m}}\geq t^\frac{1-\chi}{m}.
\end{equation}We now estimate $\rho_{T}(t)=C_{g_1,g_2}\sqrt{R_T(t)}$ from below. Recall that $R_T(t)$ is monotonically increasing in $(0,\theta_T)$  and  $\theta_T\asymp T$ (see the proof of Lemma \ref{lem:roots lemma}). Since $\al_{\de}(T,\e_1)\text{ and }\la(T,\e_2)$ converge to zero as $T\to\infty$, have that $(\al_{\de}(T,\e_1),\la(T,\e_2))\subseteq(0,\theta_T)$. Thus $\forall t\in(\al_{\de}(T,\e_1),\la(T,\e_2))$ it holds \begin{align}
\sqrt{R_T(t)}\geq\sqrt{R_T(\al_\de(T,\e_1))}=&\sqrt{-\det(\text{g}_1\text{g}_2)^{-2} \al_{\de}(T,\e_1)^{\frac{2}{m}+2}+\al_{\de}(T,\e_1)^{\frac{2}{m}}T^2-\|\text{g}_1 \text{g}_2\|^2}\nonumber  
\end{align}
We now determine the asymptotics of the terms below the above square-root. Upon recalling the definition of $\al_{\de}$ in \eqref{eq:defintion of a_delta} and using the estimate \eqref{eq:asymp for a when A small then sqrt T}, we see that 
$$\al_{\de}(T,\e_1)^{\frac{2}{m}+2}\asymp \al_{T}^{\frac{2}{m}+2}\asymp \frac{1}{T^{2+2m}}$$
For the second term, 
 $$\al_{\de}(T,\e_1)^\frac{2}{m}=\left(\al_T^{\frac{1}{m}}+\frac{\|\text{g}_1 \text{g}_2\|}{T^{1+\e_1}}\right)^2\underbrace{=}_{\eqref{eq:asymp for a when A small then sqrt T}}\|\text{g}_1 \text{g}_2\|^2\left(\frac{1}{T^2}+\frac{2}{T^{2+\e_1}}\right)+O\left(\frac{1}{T^{2+2\e_1}}\right),$$which implies that 
$$\al_{\de}(T,\e_1)^{\frac{2}{m}}T^2-\|\text{g}_1 \text{g}_2\|^2\asymp \frac{1}{T^{\e_1}}$$
Then, 
\begin{equation}\label{eq:asymp for lefthand side}
    \sqrt{R_T(\al_\de(T,\e_1))}\asymp  \frac{1}{T^{\e_1/2}}.
\end{equation}
We are free to choose $\e_1$ and $\e_2$ in the interval $(0,1)$ as we wish. In particular, we as well may assume that $$\frac{\e_2(1-\chi)}{m+1}>\e_1/2.$$With the latter choice,  we obtain our claim by the estimates \eqref{eq:asymp for righthand side} and \eqref{eq:asymp for lefthand side}.
\end{proof}
We are now ready to prove $\eta(\infty)=0.$ 
%Let $\alpha=t^{0.3}$ and consider the ball $D_{q,t}(\alpha)$ which has the same center with $D$ but with radius $\alpha$. For sufficiently small 
By the above Lemma \ref{lem:the domain Dt,T has a ball of sufficiently large radius in it} and Lemma \ref{our lemma expansion inequality for applying shah dichotomy}, we conclude that the first outcome in Theorem \ref{Shah dichotomy theorem} for the translated ellipsoids $D=D_{T,t}$ fails in the range $t\in(\al_\de(T,\e_1),\la(T,\e_2))$ for all $T$ large enough. Here $\e_1,\e_2$ are fixed such that Lemma \ref{lem:the domain Dt,T has a ball of sufficiently large radius in it} holds.
Thus, by the second outcome of Theorem \ref{Shah dichotomy theorem}, for an  $\e>0$  there is a compact $C\subset G/\Ga$  such that
\begin{equation}\label{eq:small proporotion of visits outside a compact C}
    \vol(v\in D_{T,t}:\varphi_t(v)\Ga \notin C)< \e \vol(D_{T,t}),~\forall t\in(\al_\de(T,\e_1),\la(T,\e_2)).
\end{equation}
Now fix $x=g_0\Ga$, and let $F\in C_c(G/\Ga)$ with $\mathbf{1}_C \le F \le 1$. As $T\to \infty$,
\begin{align*}
\mu_{T,g_1,g_2}^\circ(F)=&\frac{1}{\mu \left(V_{T}[g_1,g_2] \right)}\int_{V_{T}[g_1,g_2]} F( h^{-1}g_0\Gamma) d\mu(h)\\
=& \frac{1}{\mu \left(V_{T}[g_1,g_2] \right)} \int_{\al_{\de}(T,\e_1)}^{\la(T,\e_2)} \int_{D_{T,t}} F \left(\varphi_t(v)\Ga \right) dv \frac{1}{t^{m+2}}dt+O(T^{-\ka})\tag{by \eqref{eq:truncated integral}}\\
\ge& \frac{1}{\mu \left(V_{T}[g_1,g_2] \right)} \int_{a_{\de}(T,\e_1)}^{\la(T,\e_2)} \int_{D_{T,t}} \mathbf{1}_C \left(\varphi_t(v)\Ga \right) dv \frac{1}{t^{m+2}}dt +O(T^{-\ka})\\
\ge& \frac{1}{\mu \left(V_{T}[g_1,g_2] \right)} \int_{a_{\de}(T,\e_1)}^{\la(T,\e_2)}  (1-\e)\vol(D_{T,t}) \frac{1}{t^{m+2}}dt+O(T^{-\ka}) \tag{by \eqref{eq:small proporotion of visits outside a compact C}}\\
=&(1-\e)\frac{1}{\mu \left(V_{T}[g_1,g_2] \right)} \int_{\al_T}^{\be_T} \vol(D_{T,t}) \frac{1}{t^{m+2}}dt +O(T^{-\ka})\tag {by Lemmata \ref{lemma for integral from a to a delta} and \ref{lem:integral from c to b}}\\
=&1-\e+O(T^{-\ka}).
\end{align*}
Therefore,
\begin{equation*}
    \liminf_{T\to \infty}\mu_{T,g_1,g_2}^\circ(F)\ge 1-\e
\end{equation*}
so that
\begin{equation*}
    \eta(\infty)\le\limsup_{T\to \infty}\mu_{T,g_1,g_2}^\circ(\text{support}(F)^c) \le \e.
\end{equation*}

Since $\e$ is arbitrary $\eta(\infty)=0$.

\subsection{Zero measure for the singular set -- concluding equidistribution}\label{proof of equidistribution}

Recall $U=\begin{bmatrix}
    I_m & 0\\
    \R^m & 1
\end{bmatrix}$. For a closed subgroup
H of $G$, denote
\begin{align*}
    &N(\text{H},U):=\{g\in G:Ug\subset g\text{H}\},\\
    &S(\text{H},U):=\cup_{\text{H}'\subsetneq \text{H}, \text{H}'\in \mathcal{H}_{\Ga}}N(\text{H}',U).
\end{align*}
Consider,
\begin{equation}
    Y:=\bigcup_{\text{H}\in \mathcal{H}_{\Ga}}N(\text{H},U)\Ga=\bigcup_{\text{H}\in \mathcal{H}_{\Ga}}[N(\text{H},U)-S(\text{H},U)]\Ga \subset G/\Ga,
\end{equation}
where $\mathcal{H}_
{\Ga}$ was defined above Theorem \ref{ratner's theorem on discreteness}.
The equality holds since for any $g\in N(\text{H},U)$, if $g$ is also in $S(\text{H},U)$, then $g$ must belong to $N(\text{H}',U)$ for some $\text{H}'\subsetneq H$ (note for Lie subgroups, this condition means $\dim \text{H}' < \dim \text{H}$) and $\text{H}'\in \mathcal{H}_{\Ga}$. Since $\text{H}'$ has strictly lower dimension, by repeating this argument we see eventually, $g$ will fall into some $N(\tilde{\text{H}}, U)-S(\tilde{\text{H}}, U)$ (with $S(\tilde{\text{H}}, U)$ possibly empty when $\tilde{\text{H}}$ has minimal dimension).

By Theorem 2.2 of \cite{Mozes1995OnTS}, in order to show that $\eta$ is the $G$-invariant probability measure $\mu_X$, it is sufficient to prove the following lemma.
\begin{lemma}\label{lem:eta of the singular set is zero}
    $\eta(Y)=0$.
\end{lemma}
    By the discreteness of $\mathcal{H}_{\Ga}$ (Theorem \ref{ratner's theorem on discreteness}), it suffices to show that for each fixed $\text{H}\in\mathcal{H}_{\Ga}$ it holds
    \begin{equation*}
        \eta([N(\text{H},U)-S(\text{H},U)]\Ga)=0.
    \end{equation*} Since $[N(\text{H},U)-S(\text{H},U)]\Ga$ is a countable union of compact subsets in $G/\Ga$ (See \cite{Mozes1995OnTS} Proposition 3.1), it suffices to show $\eta(C)=0$ for any compact subset $C$ of $[N(\text{H},U)-S(\text{H},U)]\Ga$.

The main tool in the proof of the latter statement is the following consequence of Proposition 5.4 in \cite{Shah1994LimitDO}. We use in the following the same notations as in Section \ref{sec:non escape of mass}.
\begin{theorem}\label{second dichotomy theorem of Shah}
    Let $d,n \in \N, \e>0, \emph{\text{H}}\in \mathcal{H}_{\Ga}$. For any compact set $C \subset [N(\emph{\text{H}}, U)-S(\emph{\text{H}},U)]\Ga$, there exists a compact set $F \subset \V_G$ such that for any neighborhood $\Phi$ of $F$ in $\V_G$, there exists a
neighborhood $\Psi$ of $C$ in $G/\Ga$ such that for any $\varphi \in \mathcal{P}_{d,n}(G)$ and a bounded open convex set
$D \subset \R^n$, one of the following holds:
\begin{enumerate}
    \item There exist $\gamma\in \Ga$ such that $\varphi(D)\gamma.p_\emph{\text{H}} \subset \Phi$.
    \item $\vol(t\in D: \varphi(t)\Ga \in \Psi)<\e \vol(D)$, where $\vol$ is the Lebesgue measure on $\R^n$.
\end{enumerate}
\end{theorem}
Fix $\e>0$ and $\text{H}\in\mathcal{H}_\Ga$. Recall that $\varphi_t(\cdot)$ which was defined in \eqref{eq:definition of q t v} is in $\mathcal{P}_{2,m}$. Let $C \subset [N({\text{H}}, U)-S({\text{H}},U)]\Ga$ be a compact set and  take a compact set $F\subset \V_G$ satisfying the outcome of Theorem \ref{second dichotomy theorem of Shah}. By the same argument as in the \ref{section 4.2} with Lemma \ref{our lemma expansion inequality for applying shah dichotomy} applied, for $t\in (\al_{\de}(T,\e_1), \la(T,\e_2))$ for all $T$ large enough, the first outcome of Theorem \ref{second dichotomy theorem of Shah} fails for $\varphi_t(D_{t,T})$. Then, for $t\in (\al_{\de}(T,\e_1), \la(T,\e_2))$ for all $T$ large,
\begin{equation}\label{eq:visits to Psi is negligble}
    \vol(v\in D_{T,t}:\varphi_t(v)\Ga \in \Psi)<\e \vol(D_{T,t}),
\end{equation}
where $C\subseteq\Psi$ is the compact neighborhood from Theorem \ref{second dichotomy theorem of Shah}. 
Now let $f\in C_c(G/\Ga)$ be such that $\mathbf{1}_C \le f \le 1$ and $\text{support}(f)\subset \Psi$, it follows that
\begin{align*}
\eta(C)\leq& \limsup_{T\to \infty}\mu_{T,g_1,g_2}^\circ(f)\\
=& \limsup_{T\to \infty} \frac{1}{\mu \left(V_{T}[g_1,g_2] \right)} \int_{\al_T}^{\be_T} \int_{D_{T,t}} f \left(\varphi_t(v)\Ga \right) dv\frac{1}{t^{m+2}}dt \\
=& \limsup_{T\to \infty} \frac{1}{\mu \left(V_{T}[g_1,g_2] \right)} \int_{\al_{\de}(T,\e_1)}^{\la(T,\e_2)} \int_{D_{T,t}} f \left(\varphi_t(v)\Ga \right) dv \frac{1}{t^{m+2}}dt \tag{by Lemmata \ref{lemma for integral from a to a delta} and \ref{lem:integral from c to b}}\\
\le& \limsup_{T\to \infty} \frac{1}{\mu \left(V_{T}[g_1,g_2] \right)} \int_{\al_{\de}(T,\e_1)}^{\la(T,\e_2)} \int_{D_{T,t}} \mathbf{1}_{\Psi} \left(\varphi_t(v)\Ga \right) dv \frac{1}{t^{m+2}}dt \\
\le& \limsup_{T\to \infty} \frac{1}{\mu \left(V_{T}[g_1,g_2] \right)} \int_{\al_{\de}(T,\e_1)}^{\la(T,\e_2)} \e \vol(D_{T,t}) \frac{1}{t^{m+2}}dt \\
\le & \e.
\end{align*}
Since $\e>0$ is arbitrary, $\eta(C)=0$ and thus $\eta(Y)=0$.

\section{Proof of Theorem \ref{equidistribution result on  G mod H}: The limiting measure on $H\backslash G$.}\label{sec:finializing the proof using GW}
\iffalse
We need to prove for all $\varphi\in C_c(X_{m,m+1})$, we have 
\begin{equation}
\lim_{T
\to\infty} \frac{1}{\#\Gamma_{T}}\sum_{\gamma\in \Ga_T}\varphi(x_{0}\cdot \gamma)=\int_{X_{m,m+1}}\varphi(x)d\tilde\nu_{x_{0}}(x)
\end{equation}
for some  measure $\tilde\nu_{x_0}$ on $X_{m,m+1}$ which depends on $x_{0}$.

The first thing we need to do is to construct a measure $\nu_{x_0}$ on $H\backslash G$. Since $G$ is unimodular while $H$ is not (the measure $\mu$ we defined is not right invariant), there is no $G$-invariant Haar measure on $H\backslash G$. 

Nevertheless, it is still possible to find a lift $Y\subset G$ of $H\backslash G$ such that $H \times Y \to G$ is a Borel isomorphism and decompose the Haar measure on $G$ to obtain a measure on $Y$. Here 
\fi

We follow Section 2.5 of \cite{Gorodnik2004DistributionOL}.
First, we provide an explicit measurable section $$\sigma:H\backslash G \to Y\subset G,$$and then we provide a measure $\nu_Y$ on $Y$ such that $dg=d\mu d\nu_Y$, where $d\mu$ is the left Haar measure on $H$ given by \eqref{eq:Haar measure on $H$} and $dg$ is the Haar measure on $G$ normalized such that $\vol(G/\Ga)=1$

To define the section,  let $\mathcal{F}_m \subset \begin{bmatrix}
\SL(m,\R) & 0\\
0 & 1
\end{bmatrix}$ denote a measurable fundamental domain of 
\begin{equation*}
   \begin{bmatrix}
    \Delta & 0\\
    0 & 1
    \end{bmatrix} \bigg\backslash 
   \begin{bmatrix}
    \SL(m,\R) & 0\\
    0 & 1
    \end{bmatrix} \bigg\slash  \begin{bmatrix}
    \SO(m,\R) & 0\\
    0 & 1
    \end{bmatrix} 
    \cong
    \Delta\backslash \SL(m,\R) /\SO(m,\R),
\end{equation*}
and consider the product $$Y:=\mathcal{F}_m\cdot \SO(m+1,\R).$$
\begin{lemma}
 The product map $H \times Y \to G$ is a Borel isomorphism. In particular $\sigma:H\backslash G\to Y$ defined by $$\sigma(Hg)=y,\text{ where $g=hy,\text{ for }h\in H,y\in Y$}$$ is a well defined section identifying measurably $H\backslash G$ with $Y$.   
\end{lemma}
\begin{proof}
  The surjectivity is clear from the block-wise Iwasawa decomposition. We only verify the injectivity here:

Suppose $h_1g_1s_1=h_2g_2s_2$ where $h_i\in H, g_i\in \mathcal{F}_m$ and $s_i\in \SO(m+1,\R)$. Then 
$$\begin{bmatrix}
* & 0 \\
* & *
\end{bmatrix}\ni g_2^{-1} h_2^{-1}h_1 g_1=s_2 s_1^{-1} \in \SO(m+1,\R),$$
and it follows from the definition of $\SO(m+1,\R)$ that both sides of the above equality must lie in $\begin{bmatrix}\SO(m,\R) & 0\\0 & 1 \end{bmatrix}$. So $h_2^{-1}h_1=g_2s_2s_1^{-1}g_1^{-1}\in \begin{bmatrix}\SL(m,\R) & 0\\0 & 1 \end{bmatrix}$. But $\begin{bmatrix}\SL(m,\R) & 0\\0 & 1 \end{bmatrix} \cap H= \begin{bmatrix}\Delta & 0\\0 & 1 \end{bmatrix}$. Hence $q g_1s_1=g_2s_2$ for some $q \in \begin{bmatrix}
\Delta & 0 \\0 & 1
\end{bmatrix}$. It follows that $g_2^{-1}q g_1=s_2 s_1^{-1}\in \begin{bmatrix}\SL(m,\R) & 0 \\0 & 1 \end{bmatrix} \cap \SO(m+1,\R)=\begin{bmatrix}\SO(m,\R) & 0 \\0 & 1 \end{bmatrix}$. Hence $g_1=g_2, s_1=s_2$ by the definition of fundamental domain, and $h_1=h_2$.  
\end{proof}

%Let $K=\{(k,k):k\in \mathcal{F}_m\cap \SO(m+1,\R) \}$. Then $\frac{\mathcal{F}_m\times \SO(m+1,\R)}{K} \cong \mathcal{F}_m \SO(m+1,\R)$

Next, let $P:=
\left\{ u_va_t\tilde{\text{g}}:\text{g}\in\SL(m,\R),~v\in\R^m,~t>0 \right\},$ where $u_v,a_t$ are given by \eqref{eq:definition of u_v a_t and qtilde}, and $\tilde{\text{g}}:=\begin{bmatrix}
    \text{g}&0\\0&1
\end{bmatrix}$. On $P$ we define the Haar measure  $dp=dv \frac{dt}{t^{m+2}} d\text{g}$, where $d\text{g}$ is the Haar measure on $\SL(m,\R)$ defined through standard Iwasawa decomposition, under which $\SO(m,\R)$ has volume $1$. We note the formula \begin{equation}
    \int_{\SL(m,\R)}f(\text{g})d\text{g}=\int_{\mathcal{F}_m}\sum_{q\in\Delta}\int_{\SO(m,\R)}f(q\text{g}\rho')d\rho' d\text{g}.\label{eq:unfolding formula on the double coset space}
\end{equation}
\begin{lemma}
  Let $d\rho$ the  Haar measure probability measure on $\SO(m+1,\R)$, and consider the measure  $\nu_Y$ on $Y$ defined by  \emph{$d\nu_Y:=d\rho d\text{g}$}. Then \begin{equation} \label{unfolding haar measure on G using section}
    dg=d\mu d\nu_Y,
\end{equation}where $\mu$ is \eqref{eq:Haar measure on $H$}, and $dg$ is the Haar measure on $G=\SL(m+1,\R)$ such that $\vol(G/\Ga)=1$.  
\end{lemma} 
\begin{proof}
  by using Theorem 8.32 of \cite{KN02}, we unfold a Haar measure on $G$  as follows
\begin{align*}
    &\int_G f(g)dg\\
   =&\int_{\SO(m+1,\R)} \int_P f(p\rho) dp d\rho \\
   =&\int_{\SO(m+1,\R)}\int_{\SL(m,\R)} \int_0^{\infty}\int_{\R^m} f(u_v a_t \tilde{\text{g}} \rho)    dv \frac{dt}{t^{m+2}}d\text{g} d\rho \tag{by definition of $dp$}\\
   =&\int_{\SO(m+1,\R)}\sum_{q \in \Delta}\int_{\mathcal{F}_m}\int_{\SO(m,\R)}  \int_0^{\infty}\int_{\R^m} f(u_v a_t \tilde{q}\tilde{\text{g}}\tilde{\rho'} \rho)    dv \frac{dt}{t^{m+2}}d\text{g}  d\rho'd\rho\tag{by formula \eqref{eq:unfolding formula on the double coset space}}\\
   =& \int_{\SO(m+1,\R)}\sum_{q \in \Delta}\int_{\mathcal{F}_m} \int_0^{\infty}\int_{\R^m}  f(u_v a_t \tilde{q} \tilde{\text{g}}\rho)    dv \frac{dt}{t^{m+2}}d\text{g}d\rho\tag{invariance of $d\rho$}\\
   =& \int_{\SO(m+1,\R)}\sum_{q \in \Delta} \int_0^{\infty}\int_{\R^m} \int_{\mathcal{F}_m}  f(u_v a_t \tilde{q} \tilde{\text{g}}\rho)  d\text{g}  dv \frac{dt}{t^{m+2}}d\rho\\
   =&\int_H \int_{\SO(m+1,\R)}\int_{\mathcal{F}_m}  f(h \tilde{\text{g}}\rho)  d\text{g} d\rho  d\mu(h)\\
%   =&\text{Vol}(\SO(m,\R))^2 \int_H \int_{\frac{\mathcal{F}_m \times S}{K}}  f(h gK) d(gK)  dh
%   =&\text{Vol}(\SO(m,\R))^2\int_H \int_{\mathcal{F}_m S}  f(h g) dg  dh\\
   =& \int_H\left( \int_{\mathcal{F}_m}\int_{\SO(m+1,\R)}  f(h \tilde{\text{g}}\rho)  d\rho d\text{g}\right)  d\mu(h), 
\end{align*}
which proves the statement.  
\end{proof}

%For g_1\in G with decomposition %$g_1^{-1}:=k\begin{bmatrix}
%\text{g}_1 & 0 \\ \text{v}_1 & \det(\text{g}_1)^{-1}
%\end{bmatrix}, k\in \SO(m+1,\R)$. 

Fix $Hg_0\in H\backslash G$. By identifying $H \backslash G$ with $Y$, we define a measure on $H \backslash G$ via
\begin{equation}\label{relation between measures on H G and Y}
    d\nu_{Hg_0}(Hg):=\omega(\sigma(Hg_0),\sigma(Hg))d\nu_{Y}(\sigma(Hg)).
\end{equation}
where $\omega(\cdot,\cdot)$ is given in \eqref{eq:alpha}. 

\vspace{3mm}
In view of Theorem 2.2 and Corollary 2.4 in \cite{Gorodnik2004DistributionOL}, an immediate consequence of Theorem \ref{The $G$-invariance of limiting measure} is the following

\begin{corollary}
Fix $g_0\in G$. For any compactly supported $\varphi \in C_c(H\backslash G)$,
\begin{equation*}
    \lim_{T\to \infty}\frac{1}{\mu(H_T)}\sum_{\gamma \in \Ga_T}\varphi(Hg_0.\gamma)= \int_{H\backslash G}\varphi(Hg)d\nu_{Hg_0}(Hg).
\end{equation*}
\end{corollary}

%This proves Theorem \ref{equidistribution result on  G mod H} with $\tilde \nu_{x_0}=c_{\Gamma}\nu_{x_0}$, where $x_0\in X_{m,m+1}$ is identified with $Hg_0\in H\backslash G$ and 

%\frac{1}{m(G/\Ga)}

Now we would like to replace the normalization factor $\mu(H_T)$ by $\# \Ga_T$.
By \cite{DRS93},
\begin{equation*}
    \lim_{T\to \infty}\frac{\vol(G_T)}{\# \Ga_T}=1.
\end{equation*}
By formula A.1.15 in \cite{DRS93} and by Proposition \ref{computation for H and V}, the limit $$L:=\lim_{T\to \infty}\frac{\mu(H_{T})}{\vol(G_T)}$$ exists. Thus we conclude
\begin{equation*}
    \lim_{T\to \infty}\frac{\mu(H_{T})}{\#\Ga_T}=\lim_{T\to \infty}\frac{\mu(H_{T})}{\vol(G_T)}\frac{\vol(G_T)}{\#\Ga_T}=\lim_{T\to \infty}\frac{\mu(H_{T})}{\vol(G_T)}=L,
\end{equation*}
which shows that $$\lim_{T\to \infty}\frac{1}{\#\Ga_T}\sum_{\gamma \in \Ga_T}\varphi(Hg_0.\gamma)= L\int_{H\backslash G}\varphi(Hg)d\nu_{Hg_0}(Hg).$$

Our goal now will be to show that $L\nu_{Hg_0}(H\backslash G)=1$. We first have the following lemma.
\begin{lemma}
  It holds that $\nu_{H}(H\backslash G)<\infty$ and   $\nu_{Hg_0}(H\backslash G)=\nu_{H}(H\backslash G)$ for all $Hg_0\in H \backslash G.$
\end{lemma}
\begin{proof}
For $g_0,y \in Y$, we consider the decompositions 
\begin{equation} \label{decompositions of g_0 and y in Y}
    g_0:=\begin{bmatrix} \text{g}_0 & 0 \\0 & 1 \end{bmatrix}\rho_0,~ y:=\begin{bmatrix} \text{g} & 0 \\0 & 1 \end{bmatrix} \rho,
\end{equation}
where $\text{g}_0,\text{g}\in\mathcal{F}_m$ and $\rho_0,\rho\in\SO(m+1,\R).$ Hence $\alpha$ in \eqref{eq:alpha} takes the form: $ \omega(g_0,y)
    = \frac{\sum_{q\in \Delta}\frac{1}{\|\text{g}_0^{-1}q \text{g}\|^{m^2}}}{\sum_{q\in \Delta}\frac{1}{\|q \|^{m^2}}}.$
Using the $\SO(m+1,\R)$-invariance of the Hilbert-Schmidt norm, we compute 
\begin{align*}\label{folding trick}
    \nu_{Hg_0}(H\backslash G)
    =&\int_{Y}\omega(g_0,y)d\nu_{Y}\\
 %   =& \vol(\SO(m+1,\R)) \int_{\mathcal{F}_m}2\pi dg\\
    =&\frac{1}{\sum_{q\in \Delta}\frac{1}{\|q \|^{m^2}}} \int_{\SO(m+1,\R)}\int_{\mathcal{F}_m}\sum_{q\in \Delta}\frac{1}{\|\text{g}_0^{-1}q \text{g}\|^{m^2}} d\text{g} d\rho \tag{definition of $d\nu_Y$}\\
    =&\frac{1}{\sum_{q\in \Delta}\frac{1}{\|q \|^{m^2}}} \int_{\SL(m,\R)}\frac{1}{\|\text{g}_0^{-1}\text{g}\|^{m^2}}  d\text{g}\tag{by formula \eqref{eq:unfolding formula on the double coset space}}\\
    =&\frac{1}{\sum_{q\in \Delta}\frac{1}{\|q \|^{m^2}}} \int_{\SL(m,\R)}\frac{1}{\|\text{g}\|^{m^2}}  d\text{g}\tag{invariance of $d\text{g}$}\\
    =&\nu_{H}(H\backslash G).
\end{align*}
It remains to verify that $\int_{\SL(m,\R)}\frac{1}{\|\text{g}\|^{m^2}}  d\text{g}<\infty.$ By monotone convergence, it's enough to prove that $\int_{\SL(m,\R)_T}\frac{1}{\|\text{g}\|^{m^2}}  d\text{g}$ converges as $T\to\infty$, where $\SL(m,\R)_T=\{\text{g}\in\SL(m,\R):\|\text{g}\|\leq T\}.$  For that we use a Fubini trick, similarly to the proof of Lemma \ref{lemma on the integral and summation in a ball}: 
\begin{align*}
    \int_{\SL(m,\R)_T}\frac{1}{\|\text{g}\|^{m^2}}  d\text{g}&=\int_{\SL(m,\R)_T}\int_0^\infty1_{\{t\leq\|\text{g}\|^{-m^2}\}} dt d\text{g}\\
    &=\int_0^\infty \int_{\SL(m,\R)_T}1_{\{t\leq\|\text{g}\|^{-m^2}\}}d\text{g} dt\\
    &=\int_0^1 \vol\{\text{g}:\|\text{g}\|\leq \text{min}(T,t^{-\frac{1}{m^2}})\} dt,\tag{using that $\|\text{g}\|\geq 1,\forall\text{g}$}
\end{align*}where  $\vol$ stands for the $\SL(m,\R)$-Haar volume. The convergence is established by using formula A.1.15 in \cite{DRS93} which gives $\vol\{\SL(m,\R)_s\}\sim s^{m(m-1)}.$
\end{proof} Finally, we prove

\begin{proposition}
    $\nu_{Hg_0}(H\backslash G)=\nu_{H}(H\backslash G)=\lim_{T\to \infty}\frac{\vol(G_T)}{\mu(H_T)}=\frac{1}{L}$.
\end{proposition}

\begin{proof}
%[Proof of the claim]\renewcommand{\qedsymbol}{\ensuremath{\#}}
    %The proof here follows from Theorem 2.3 in \cite{Gorodnik2004DistributionOL} and its proof.  
Since $\nu_H$ is a Radon measure, for any $\e>0$, we can choose $f_{\e}\in C_c(H\backslash G)$ with support $B_{\e}$  such that 
    \begin{equation*}
        \int_{Y}|f_{\e}(Hy)-1|\omega(e,y)d\nu_Y(y)=\int_{H\backslash G}|f_{\e}(Hg)-1|d\nu_H(Hg)\le \e.
\end{equation*}
As in \cite{Gorodnik2004DistributionOL}, we observe that    \begin{align*}
        \frac{1}{\mu(H_T)}\int_{G_T}|f_{\e}(Hg)-1|dg
        =&\frac{1}{\mu(H_T)}\int_Y\int_{\{h:\|hy\|<T\}}|f_{\e}(Hy)-1|d\mu(h)d\nu_Y(y)\\
        =&\frac{1}{\mu(H_T)}\int_Y |f_{\e}(Hy)-1|\mu(\{h:\|hy\|<T\})d\nu_{Y}(y)\\
        =&\int_Y |f_{\e}(Hy)-1| \frac{\mu(H_T[e,y])}{\mu(H_T)}d\nu_{Y}(y)
    \end{align*}
Recall that $\lim_{T\to\infty}\frac{\mu(H_T[e,y])}{\mu(H_T)}=\omega(e,y).$
We will use below the dominated convergence theorem, and for that we will now show that the integrand $\frac{\mu(H_T[e,y])}{\mu(H_T)}$ is bounded by a function in $L^1(Y)$ for large $T$. 
By \eqref{eq:bound on measure of Vga T} and by \eqref{eq:haar measure of skewed H ball as a sum of skewed balls on connected component}, we see  
\begin{equation*}
\frac{\mu(H_T[e,y])}{\mu(H_T)} \ll \sum_{q\in \Delta}\frac{1}{\|q \text{g}\|^{m^2}},\text{ where $y=\begin{bmatrix} \text{g} & 0 \\0 & 1 \end{bmatrix} \rho,~\text{g}\in \SL(m,\R),\rho\in\SO(m+1,\R)$}.
\end{equation*}
Since $\int_{\SL(m,\R)}\frac{1}{\|\text{g}\|^{m^2}}d\text{g}$ converges, it follows that $\Psi(y):=\sum_{q\in \Delta}\frac{1}{\|q \text{g}\|^{m^2}}\in L^{1}(Y)$ (more precisely, to see that this is a $L^1$ function, use \eqref{eq:unfolding formula on the double coset space}).

By the dominant convergence theorem, the second term satisfies
\begin{equation*}
     \lim_{T\to \infty}\int_{Y} |f_{\e}(Hy)-1| \frac{\mu(H_T[e,y])}{\mu(H_T)}d\nu_{Y}(y)= \int_{Y} |f_{\e}(Hy)-1| d\nu_{H}(y)\le \e.
\end{equation*} 
Therefore, by triangular inequality
\begin{align*}
    &\limsup_{T\to \infty}\left|\frac{\vol(G_T)}{\mu(H_T)}-\nu_H(H\backslash G)\right|\\
    \le & \limsup_{T\to \infty}\left|\frac{\vol(G_T)}{\mu(H_T)}-\frac{1}{\mu(H_T)}\int_{G_T}f_{\e}(Hg)dg\right|\\
    &+\limsup_{T\to \infty}\left|\frac{1}{\mu(H_T)}\int_{G_T}f_{\e}(Hg)dg-\int_{H\backslash G}f_{\e}(Hg)d\nu_H(g) \right|\\
    &+\limsup_{T\to \infty}\left|\int_{H\backslash G}f_{\e}(Hg)dg-\nu_H(H\backslash G) \right|\\
    \le& \e+0+\e \tag{the middle term vanishes because of Theorem 2.3, \cite{Gorodnik2004DistributionOL}}
\end{align*}
Now let $\e\to 0$, and this finishes the proof.
\end{proof}
Therefore, $\tilde\nu_{Hg_0}:=L\nu_{Hg_0}$ is a probability measure, and we conclude 
\begin{equation*}
    \lim_{T\to \infty}\frac{1}{\# \Ga_T}\sum_{\gamma \in \Ga_T}\varphi(Hg_0.\gamma)= \int_{H\backslash G}\varphi(Hg)d\tilde \nu_{Hg_0}(g).
\end{equation*}

\section*{Acknowledgements}
The authors would like to thank Nimish Shah for helpful discussions and generous sharing of his ideas leading to the proof of Lemma \ref{our lemma expansion inequality for applying shah dichotomy}. We thank Uri Shapira for many helpful discussions and encouragement. We thank Alex Gorodnik and Barak Weiss for their comments on a previous version of our manuscript. We would also like to thank Osama Khalil for discussions related to this work. Finally, we would like to thank an anonymous referee for their careful reading and valuable feedback which helped us improve many places in this paper. This work has received funding from the European Research Council (ERC) under the European Union’s Horizon 2020 Research and Innovation Program, Grant agreement No. 754475.

\printbibliography[
heading=bibintoc,
title={Bibliography}
]

@article{DRS93,
  title={Density of integer points on affine homogeneous varieties},
  author={William Duke and Ze{\'e}v Rudnick and Peter Sarnak},
  journal={Duke Math. J.},
  year={1993},
  volume={71},
  pages={143-179}
}

@article{Dani1980OrbitsOE,
  title={Orbits of euclidean frames under discrete linear groups},
  author={S.G.Dani and S. Raghavan},
  journal={Israel J. Math.},
  year={1980},
  volume={36},
  pages={300-320}
}

@article{Kleinbock2006DIRICHLETSTO,
  title={Dirichlet's theorem on diophantine approximation and homogeneous flows},
  author={Dmitry Kleinbock and Barak Weiss},
  journal={J. Mod. Dyn.},
  year={2006},
  volume={2},
  pages={43-62}
}

@article{Oh05,
author = {Oh, Hee},
year = {2005},
month = {07},
pages = {635-653},
title = {Lattice action on finite volume homogeneous spaces},
volume = {42},
journal = {J. Korean Math. Soc},
@doi = {10.4134/JKMS.2005.42.4.635}
}

@misc{gorodnik2022stationary,
      title={Stationary measures for $\mathrm{SL}_2(\mathbb{R})$-actions on homogeneous bundles over flag varieties}, 
      author={Alexander Gorodnik and Jialun Li and Cagri Sert},
      year={2022},
      eprint={2211.06911},
      archivePrefix={arXiv},
      primaryClass={math.DS}
}

@article{Sargent2017DynamicsOT,
  title={Dynamics on the space of 2-lattices in 3-space},
  author={Oliver Sargent and Uri Shapira},
  journal={Geom. Funct. Anal.},
  year={2017},
  volume={29},
  pages={890-948}
}

@article{Shah2000OnAO,
  title={On actions of epimorphic subgroups on homogeneous spaces},
  author={Nimish A. Shah and Barak Weiss},
  journal={Ergodic Theory Dynam. Systems},
  year={2000},
  volume={20},
  pages={567 - 592}
}

@article{symmetricdifferencesch2010,
author = {Schymura, Daria},
year = {2010},
month = {10},
pages = {},
title = {An upper bound on the volume of the symmetric difference of a body and a congruent copy},
journal = {Adv. Geom.},
@doi = {10.1515/advgeom-2013-0029}
}

@article{GN14duality,
  title={Ergodic theory and the duality principle on homogeneous spaces},
  author={Alexander Gorodnik and Amos Nevo},
  journal={Geom. Funct. Anal.},
  year={2014},
  volume={24},
  pages={159–244}
}

@article{Mozes1995OnTS,
  title={On the space of ergodic invariant measures of unipotent flows},
  author={Shahar Mozes and Nimish A. Shah},
  journal={Ergodic Theory Dynam. Systems},
  year={1995},
  volume={15},
  pages={149 - 159}
}

@article{Shah1994LimitDO,
  title={Limit distributions of polynomial trajectories on homogeneous spaces},
  author={Nimish A. Shah},
  journal={Duke Math. J.},
  year={1994},
  volume={75},
  pages={711-732}
}

@article{Ratner91a,
 author = {Marina Ratner},
 journal = {Ann. of Math.},
 number = {3},
 pages = {545--607},
 publisher = {Ann. of Math.},
 title = {On Raghunathan's Measure Conjecture},
 volume = {134},
 year = {1991}
}

@article{Dani1993LimitDO,
  title={Limit distributions of orbits of unipotent flows and values of quadratic forms},
  author={S. G. Dani and G. A. Margulis},
journal = {Adv. Soviet Math., vol. 16, Amer. Math. Soc., Providence,
RI},
  year={1993}
}

@article{Shah1996LimitDO,
  title={Limit distributions of expanding translates of certain orbits on homogeneous spaces},
  author={Nimish A. Shah},
  journal={Proceedings of the Indian Academy of Sciences - Mathematical Sciences},
  year={1996},
  volume={106},
  pages={105-125}
}

@article{Gorodnik2003LatticeAO,
  title={Lattice action on the boundary of $\text{SL}(n,\mathbb R )$},
  author={Alexander Gorodnik},
  journal={Ergodic Theory Dynam. Systems},
  year={2003},
  volume={23},
  pages={1817 - 1837}
}

@article{Gorodnik2004DistributionOL,
  title={Distribution of lattice orbits on homogeneous varieties},
  author={Alexander Gorodnik and Barak Weiss},
  journal={GAFA Geom. Funct. Anal.},
  year={2004},
  volume={17},
  pages={58-115}
}

@Book{KN02,
 author = { Anthony W. Knapp},
 publisher = {Springer},
 title = {Lie Groups: Beyond an Introduction, 2nd Edition},
 year = {2002}
}

@article {Roelcke56,
    AUTHOR = {Roelcke, Walter},
     TITLE = {\"{U}ber die {V}erteilung der {K}lassen eigentlich
              assoziierter zweireihiger {M}atrizen, die sich durch eine
              positiv-definite {M}atrix darstellen lassen},
   JOURNAL = {Math. Ann.},
    VOLUME = {131},
      YEAR = {1956},
     PAGES = {260--277},
      
}

@article {Schmidt98,
    AUTHOR = {Schmidt, Wolfgang M.},
     TITLE = {The distribution of sublattices of {${\bf Z}^m$}},
   JOURNAL = {Monatsh. Math.},
    VOLUME = {125},
      YEAR = {1998},
    NUMBER = {1},
     PAGES = {37--81},
}

@article {Maass59,
    AUTHOR = {Maass, Hans},
     TITLE = {\"{U}ber die {V}erteilung der zweidimensionalen {U}ntergitter
              in einem euklidischen {G}itter},
   JOURNAL = {Math. Ann.},
    VOLUME = {137},
      YEAR = {1959},
     PAGES = {319--327},
}

@book{mar_thesis,
	Author = {Grigoriy A. Margulis},
	Edition = {1},
	Publisher = {Springer Berlin, Heidelberg},
	Series = {Springer Monographs in Mathematics},
	Title = {On Some Aspects of the Theory of Anosov Systems},
	Year = {2004}}

@book {GN_book_ergodic_theory,
    AUTHOR = {Gorodnik, Alexander and Nevo, Amos},
     TITLE = {The ergodic theory of lattice subgroups},
    SERIES = {Ann. of Math. Stud.},
    VOLUME = {172},
 PUBLISHER = {Princeton University Press, Princeton, NJ},
      YEAR = {2010},
     PAGES = {xiv+121},
}

@article {Led99,
    AUTHOR = {Ledrappier, Fran\c{c}ois},
     TITLE = {Distribution des orbites des r\'{e}seaux sur le plan r\'{e}el},
   JOURNAL = {C. R. Acad. Sci. Paris S\'{e}r. I Math.},
  FJOURNAL = {Comptes Rendus de l'Acad\'{e}mie des Sciences. S\'{e}rie I.
              Math\'{e}matique},
    VOLUME = {329},
      YEAR = {1999},
    NUMBER = {1},
     PAGES = {61--64},
}
\end{document}